\documentclass[12pt,letterpaper]{amsart}
\usepackage{ amsmath,amsthm, amsbsy,amsfonts,amssymb, txfonts, shuffle}
\usepackage[normalem]{ulem}
\usepackage{stmaryrd,mathrsfs,graphicx, supertabular, amscd, tikz-cd, tikz}
\tikzcdset{row sep/my size/.initial=6mm}
\usepackage{stackrel}
\usetikzlibrary{matrix,arrows,decorations.pathmorphing}

\usepackage{ulem}

\usepackage[active]{srcltx}
\allowdisplaybreaks
\numberwithin{equation}{section}

\usepackage[
	hypertexnames=false,
	hyperindex,
	pagebackref,
	breaklinks=true,
	bookmarks=false,
	colorlinks,
	linkcolor=blue,
	citecolor=red,
	urlcolor=blue,
]{hyperref}
\usepackage{hyperref}

\usepackage{tikz, tikz-cd}
\usetikzlibrary{matrix,calc,positioning,arrows,decorations.pathreplacing,decorations.markings,patterns} 
\tikzset{->-/.style={decoration={
			markings,
			mark=at position #1 with {\arrow{>}}},postaction={decorate}}} 
\tikzset{-<-/.style={decoration={
					markings,
					mark=at position #1 with {\arrow{<}}},postaction={decorate}}}

\def\BB{{\mathbb B}}

\def\CC{{\mathbb C}}

\def\HH{{\mathbb H}}

\def\Kbold{{\mathbf K}}

\def\MM{{\mathbb M}} 
\def\Nbold{{\mathbf N}}

\def\QQ{{\mathbb Q}} 
\def\RR{{\mathbb R}} 
 
\def\TT{{\mathbb T}}

\def\ZZ{{\mathbb Z}}

\def\Ibold{\mathbf {I}}
\def\Jbold{\mathbf {J}}

\def\uNbold{{\underline{\Nbold}}}

\def\g{\gamma}

\def\tri{\triangle}

\def\Ccal{{\mathcal C}}

\def\Ecal{{\mathcal E}}

\def\Hcal{{\mathcal H}} 
\def\Ical{{\mathcal I}} 
 
\def\Kcal{{\mathcal K}}

\def\Scal{{\mathcal S}}

\def\Xcal{{\mathcal X}}

\def\sEcal{{\widehat{\Ecal}}}

\def\uk{\underline{k}}
\def\ul{\underline{l}}

\def\vfrak{{\mathfrak{v}}}

\def\mfrak{{\mathfrak{m}}}

\def\bfrak{\mathfrak{b}}
\def\dfrak{\mathfrak{d}}

\def\Sfrak{\mathfrak{S}}

\def\pt{{\scriptscriptstyle\bullet}}

\newcommand\conf{\mathscr{C}\!\mathit{onf}\!}
\newcommand\cfg{\mathit{cfg}}

\newcommand\id{\operatorname{id}}

\newcommand\gr{\operatorname{gr}}

\newcommand\Hom{\operatorname{Hom}}

\newcommand\im{\operatorname{Im}}

\newcommand\sign{\operatorname{sign}}

\newcommand\Sp{\operatorname{Sp}}

\newcommand\nc{{\operatorname{nc}}}
\newcommand\aug{{\operatorname{aug}}}

\newcommand\wt{{\operatorname{wt}}}
\newcommand\comp{{{comp}}}

\newtheorem{atheorem}{Theorem}

\newtheorem{theorem}{Theorem}[section]
\newtheorem{lemma}[theorem]{Lemma}
\newtheorem{proposition}[theorem]{Proposition}
\newtheorem{propositionconvention}[theorem]{Proposition/Convention}
\newtheorem{corollary}[theorem]{Corollary}

\newtheorem{definition}[theorem]{Definition}

\theoremstyle{remark}
\newtheorem{example}[theorem]{Example}

\newtheorem{remark}[theorem]{Remark}

\title{Surface configuration kernels}
\author{Andreas Stavrou}
\email{andreasstavrou@uchicago.edu}

\address{Mathematics Department, University of Chicago}

\begin{document}
\begin{abstract} Let $\Sigma_{g,*}$ be a once-punctured oriented surface of genus $g$.  
We study the action of the mapping class group $\Gamma_{g,*}$ on the $n^{th}$ rational cohomology of the configuration space $\conf_n(\Sigma_{g,*})$ of injections $\{1,\ldots, n\}\hookrightarrow \Sigma_{g,*}$, and compare the kernel of this action $J_{g,*}^\cfg(n)$ with the $n^{th}$ Johnson subgroup $J_{g,*}(n)$. We find high-rank abelian subgroups in the quotient $J_{g,*}^\cfg(n)/J_{g,*}(n)$ arising from the higher Johnson images and from symplectic representation theory. In particular we refute a conjecture due to Bianchi--Miller--Wilson.
\end{abstract}

\maketitle

\section{Introduction}
For a topological space $X$ and an integer $n\ge 0$, the \textit{configuration space } $\conf_n(X)$ is the space of injections $\{1,\ldots, n\}\hookrightarrow X$,  topologised as a subspace of $X^n$. We will be interested in $\conf_n(\Sigma_{g,*})$, where $\Sigma_{g,*}$ is the complement of a single point $*$ in the closed oriented surface $\Sigma_g$ of genus $g$. The \textit{mapping class group} $\Gamma_{g,*}$ is the group of orientation preserving homeomorphisms of $\Sigma_{g,*}$ up to isotopy, and it naturally acts on the rational cohomology $H^*(\conf_n(\Sigma_{g,*});\QQ)$. Our aim is to understand the kernels $J^\cfg(n)$ of the action \begin{equation}\label{eq:introkeyaction}
    \Gamma_{g,*}\curvearrowright H^n(\conf_n(\Sigma_{g,*});\QQ).
\end{equation}

These are closely related with the \textit{Johnson subgroups} $J_{g,*}(n)$ of $\Gamma_{g,*}$, defined as the kernels of $\Gamma_{g,*}\curvearrowright \pi/\pi^{(n+1)}$, where $\pi=\pi_1(\Sigma_g,*)$ and $\pi^{(n+1)}$ is the $(n+1)$th commutator subgroup defined iteratively by $\pi^{(0)}=\pi$ and  $\pi^{(i+1)}=[\pi^{(i)},\pi]$. A theorem of Bianchi, Miller and Wilson \cite{BMW} implies the inclusion $J_{g,*}(n)\subset J_{g,*}^\cfg(n)$, while a conjecture they made predicts the quotient $J_{g,*}(n)/J_{g,*}^\cfg(n)$ to be finite. 
In our earlier work with Looijenga  \cite{LooijengaStavrou25}, we disproved this conjecture in the case $n=4$. In this sequel, we expand this analysis to arbitrary $n$ and determine the quotient $J_{g,*}(n)/J_{g,*}^\cfg(n)$ more precisely.

We access $J^\cfg(n)$ by studying the representation $\Ical^\cfg_n$ from \cite{LooijengaStavrou25} obtained as follows. The group algebra $\QQ\pi$ is filtered by powers of the augmentation ideal $\Ical$, the kernel of the augmentation homomorphism $\aug: \QQ\pi\to \QQ$ sending every $\g\in \pi$ to $1\in \QQ$. There is then a $\Gamma_{g,*}$-equivariant
homomorphism
\begin{equation}\label{eq:tri_nDef}
\tri^n:\QQ\pi/\Ical^{n+1}\to H^n(\conf_n(\Sigma_{g,*})),    
\end{equation}
whose image $\Ical^\cfg_n$ has kernel precisely $J_{g,*}^\cfg(n)$. Since by a theorem of Labute \cite{Labute}, the kernel of the $\Gamma_{g,*}\curvearrowright\QQ\pi/\Ical^{n+1}$ is $J(n)$, the desired quotient $J_{g,*}^\cfg(n)/J_{g,*}(n)$ is completely controlled by the map $\tri^n$. 

The augmentation filtration of $\QQ\pi/\Ical^{n+1}$ allows us to tautologically filter the map $\tri^n$ and the image $\Ical^\cfg_n$. Each graded piece $\gr^\Ical_k\QQ\pi/\Ical^{n+1}=\Ical^k/\Ical^{k+1}$, for $k\le n$, is described by Labute as a natural quotient of $H^{\otimes k}$ where $H=H_1(\Sigma_{g,*};\QQ)$. The action of $\Gamma_{g,*}$ on $H$ is via the symplectic group $\Sp_{2g}(\ZZ)$, and we will be able to give a very nice description of $\ker\gr^\Ical_\pt\tri^n$ in terms of symplectic representation theory.

The group $\Sp_{2g}(\ZZ)$ has a well-understood class of representations over $\QQ$ called \textit{algebraic}, which are direct sums of subquotients of $H^{\otimes n}$ for $n\ge 0$. Algebraic representations decompose as direct sums of irreducibles, each of which has a \textit{weight}: the least $n\ge 0$ for which it embeds in $H^{\otimes n}$. So for each algebraic representation $V$ and $k\ge 0$, there is a well-defined subspace $V^{\le k}$ spanned by all irreducibles of weight $\le k$. With this language, we can now describe the degree-$n$ graded piece of $\ker\tri^n$. 
\begin{atheorem}\label{athm:bottompartLambda}
    The kernel of $\gr^\Ical_n\tri^n$ 
    is the weight $\le n-2$ part of $\Ical^n/\Ical^{n+1}$. In other words, only the top, weight-$n$ part of $\Ical^n/\Ical^{n+1}$ survives in $\Ical^\cfg_n$.
\end{atheorem}
\begin{remark}
    This is an improvement on  Corollary 3.5 of \cite{LooijengaStavrou25} where we had only proven an inclusion.
\end{remark}

We can already draw consequences for the quotient $J_{g,*}^{\cfg}(n)/J_{g,*}(n)$ expressed in terms of the higher Johnson images $J_{g,*}(k)/J_{g,*}({k+1})$. We note that these latter groups are only known for small $k$, with upper and lower bounds for large $k$. It is known, however, for all $k$ that after tensoring with $\QQ$, they become algebraic representations, and so it makes sense to use the notation $(J_{g,*}(k)/J_{g,*}({k+1}))^{\le w}$ for the part of $J_{g,*}(k)/J_{g,*}({k+1})$ that lands in weight $\le w$ after this tensoring. With this notation, we bound below the contribution of $J_{g,*}(n-1)$ to $J_{g,*}^\cfg(n)$.
\begin{atheorem}\label{athm:J_n-1/J_n}
    The intersection of $J_{g,*}^\cfg(n)/J_{g,*}(n)$ with $J_{g,*}({n-1})/J_{g,*}(n)$ contains the subgroup $(J_{g,*}({n-1})/J_{g,*}(n))^{\le n-3}$.
\end{atheorem}
\begin{remark}
    Table 1 of Morita--Sakasai--Suzuki \cite{MorSakSuzSympInvariantLie} provides a decomposition of $J_{g,*}(n-1)/J_{g,*}(n)$ into irreducibles for $n\le 7$, showing that the rank of $(J_{g,*}({n-1})/J_{g,*}(n))^{\le n-1}$ is non-zero for $n\ge 5$ and grows rapidly with $n$.
\end{remark} 

The last theorem is a mere corollary of our main result. In fact, we are able to describe $\ker\gr^\Ical_\pt\tri^n$ completely if the genus $g$ is large, and bound it below otherwise.
\begin{atheorem}[Theorem \ref{thm:assgradedIcalcfg}]\label{athm:groupring}
    For $1\le k\le n$, the surjection $\gr^\Ical_k\QQ\pi/\Ical^{n+1}\twoheadrightarrow \gr^\Ical_k\Ical^\cfg_n$ contains $(\gr^\Ical_k\QQ\pi/\Ical^{n+1})^{\le 3k-2(n+1)}$ in its kernel. If $g\ge k$, this containment is an equality.
\end{atheorem}

As a consequence, we find that not only $J_{g,*}(n-1)$ but many  lower Johnson subgroups $J_{g,*}(k)$ for  $k<n$ intersect $J^\cfg_{g,*}(n)$ non-trivially. For each such $k$, there exists a subspace $$K_k\subset J_{g,*}({k})/J_{g,*}(k+1)$$ which lifts to $J^\cfg_{g,*}(n)/J_{g,*}(n)$, and the various $K_k$ organise this quotient into successive layers as follows.  Starting with $C_n:=J^\cfg_{g,*}(n)/J_{g,*}(n)$ the intersection with $J_{g,*}(n-1)/J_{g,*}(n)$ is $K_{n-1}$; quotienting by $K_{n-1}$ yields $C_{n-1}$, whose intersection  with $J_{g,*}(n-2)/J_{g,*}(n-1)$ yields $K_{n-2}$, and so on. Representation-theoretic considerations force us to interpret these relations only up to finite index and finite cokernel. This leads to the following description of $J^\cfg_{g,*}(n)/J_{g,*}(n)$ as a \textit{virtual cofiltration}.

\begin{atheorem}[Theorem \ref{thm:virtualcofiltration}]\label{athm:johnsonquotient}
    The quotient $J_{g,*}^\cfg(n)/J_{g,*}(n)$ has a virtual cofiltration, that is, a sequence of group homomorphisms
    \begin{equation*}
        \begin{tikzcd}
            J_{g,*}^\cfg(n)/J_{g,*}(n)=C_n\rar["f_{n}"] & C_{n-1}\rar["f_{n-1}"] & C_{n-2}\rar["f_{n-2}"] & \cdots \rar["f_2"] & C_1 
        \end{tikzcd}
        \end{equation*}
    where each $f_k$ has finite cokernel, and where each kernel $\ker f_k$ is an abelian group containing a copy of
    $(J_{g,*}({k-1})/J_{g,*}(k))^{\le 3k-2(n+1)-1}.$
    Furthermore, if $g\ge \max(n,k+1)$, this containment is of finite index.
\end{atheorem}
For $k$ close to $n$, we then obtain high-rank free abelian contributions, these kernels are large, whereas for $k<(2(n+1)-1)/3$ we get nothing. 
We remark that we do not know whether the assumption $g\ge k$ is necessary or is an artifact of our method.

\subsection{$\QQ$-coefficients}
We warn the reader that while in  \cite{LooijengaStavrou25} we worked integrally, here we work throughout with homology over $\QQ$ instead. This means, for example, that $\Ical^\cfg_n$ corresponds to the rationalisation of the identically denoted abelian group in \cite{LooijengaStavrou25}, and that  our $J^\cfg(n)$ contains the identically denoted group of the prequel with finite index but may not be equal to it. We preferred to work over $\QQ$ to benefit from the semi-simplicity of symplectic representations, but found it tedious to carry the $\otimes \QQ$ notation throughout. However, this should cause little confusion as our theorems on $J^\cfg(n)$ are already stated in terms of ``finite index inclusions''.

\subsection{Structure of the paper} In Section \ref{sec:recollections}, we recollect the modules  $\Hcal_n(U)$ from \cite{LooijengaStavrou25}, recall the notion of gr-algebraic representations and prove Theorem \ref{athm:bottompartLambda}. 
In Section \ref{sec:keroftri^n}, we study the kernel of $\tri^n$ and employ symplectic representation theory to describe an associated graded for the image thus giving Theorem \ref{athm:groupring}. The results of this section depend on Theorem \ref{thm:independence_r+s=n} which is proven separately in Section \ref{sec:geometricconstructions} using a rather elaborate geometric constructions.
Finally, in Section \ref{sec:configurationkernels}, we translate the results on $\Ical^\cfg_n$ to results about $J^\cfg_{g,*}(n)$.
Although this is a sequel of \cite{LooijengaStavrou25}, our aim is to be self-contained.

\subsection{Acknowledgments} The author would like to thank Eduard Looijenga for the conversations which brought this paper into being. Thanks also go to Oscar Randal-Williams, as the germs of these ideas were already planted in Cambridge. The author would finally like to thank Dick Hain for a helpful conversation during his visit at the University of Chicago. 

\section{The deepest graded piece of $\tri^n$}\label{sec:recollections}
\subsection{Recollections from \cite{LooijengaStavrou25}}\label{subsec:recollections} We now recall some notations, definitions and statements from the prequel, that turn the study of $J^\cfg(n)$ into the study of the map $\tri^n$ from \eqref{eq:tri_nDef}.
\subsubsection{On partitions and orders}
We fix these notations for the rest of this paper. 

In general $N$ will be a set of size $n$. If it is equipped with a total order $\prec$, it will be denoted by the boldface $\Nbold$. 

Given two totally ordered sets $\Ibold,\Jbold$, the concatenation  $\Ibold\Jbold$ is the disjoint union $I\sqcup J$ equipped with the total order inherited from $\Ibold$ and $\Jbold$ and which makes $I\prec J$. 

A \textit{partition-with-orders} $\uNbold=(\Nbold_1,\ldots, \Nbold_r)$ of a set $N$  is a partition of $N$ into a disjoint union $N=N_1\sqcup \cdots\sqcup N_r$ and each $N_i$ is equipped with a total ordering. We say $\uNbold$ is \textit{nowhere empty} if all the parts $N_i$ are non-empty. For example, the notation $\uNbold=((3,1),(2,5,4))$ is a nowhere empty partition of the set $[5]$ into two parts; the orderings are understood to be $3\prec 1$ and $2\prec 5\prec 4$. 
\subsubsection{The functor $\Hcal_N(U)$ }
We generalise $\conf_N(X)$ to the case $N$ is a finite set of size $n$ to be the space of all injections $N\hookrightarrow X$ topologised as a subspace of the function space $X^N$.

If $(X,*)$ is a based space with $U=X\setminus \{*\}$, then the complement of $\conf_n(U)$ in $X^n$ is the space $$D_n(X,*)=\{f:N\to X|f \text{ not injective or } *\in \im f\}.$$ The relative homology $$\Hcal_N(X,*):=H_n(X^N,D_N(X,*);\QQ)$$
was studied by Moriyama and Bianchi--Miller--Wilson \cite{Moriyama, BMW} in the context of surfaces. 

In \cite{LooijengaStavrou25}, we focus at the case when $X$ is the one-point compactification of $U$, in which case  this homology 
is isomorphic to to the Borel--Moore homology $H_n^{BM}(\conf_N(U);\QQ)$, and denoted it  $\Hcal_N(U)$ instead. If furthermore $U$ is an orientable manifold of dimension $d$ with no closed component, then Poincar\'e-duality gives a $\Gamma_{g,*}$-equivariant isomorphism $\Hcal_n(U)\cong H^{(d-1)n}(\conf_n(U);\QQ)$.  The facts we present in the next sections on $\Hcal_N(U)$ were all proven in \cite{LooijengaStavrou25}.

\subsubsection{The maps $\tri^\Nbold_U$}
 Denote the $\pi_U=\pi_1(X,*)$ and $\Ical_U$ the augmentation ideal of $\QQ\pi_U$. Let also  $\prec$ be any total order on $N$, and denote $\Nbold=(N,\prec)$. 
Viewing an element $\g\in \pi_U$ as a map $\g:[0,1]\to X$, we can take the $N^{th}$ power $\g^N:[0,1]^N\to X^N$ and restrict it to the $N$-simplex $$\tri^\Nbold=\{f:N\to [0,1]|f \text{ non-decreasing}\}.$$ 
The boundary lands in $D_N(X,*)$, so we obtain a relative homology class $[\g^n|_{\tri^\Nbold}]\in \Hcal_n
(U)$, which we denote by $\tri^\Nbold_U(\g)$. This extends linearly as a map from $\QQ\pi_U$ which vanishes on $\Ical_U^{n+1}$ giving us a map
\begin{equation}\label{eq:triNboldU}
    \tri^\Nbold_U:\QQ\pi_U/\Ical_U^{n+1}\to \Hcal_N(U).
\end{equation}
We also set $\tri^{\emptyset}(\g)=1$. Note that if $n>0$, $\tri^\Nbold_U(1)$ always vanishes, and so $\tri^\Nbold_U$ can be viewed as a map from $\Ical_U/\Ical_U^{n+1}$, instead, with same image. We write simply $\tri^n_U$ to imply the set $[n]$ with the usual order.

\subsubsection{Multiplicative structure}
The maps $\tri^\Nbold$ are multiplicative in the following sense.
Taking the disjoint union of two sets $N_1,N_2$ gives a homeomorphism $X^{N_1}\times X^{N_2}\to X^{N_1\sqcup N_2}$ and, in turn, a product
\begin{equation}
\times: \Hcal_{N_1}(U)\otimes \Hcal_{N_2}(U)\to \Hcal_{N_1\sqcup N_2}(U).    
\end{equation}
The \textit{decomposition formula} gives, for any $\g_1,\g_2\in \pi$, 
\begin{equation}\label{eq:decompositionformula}
    \tri^\Nbold(\g_1\g_2)=\sum_{\Nbold_1\Nbold_2=\Nbold}\tri^{\Nbold_1}(\g_1)\times \tri^{\Nbold_2}(\g_2),
\end{equation}
where the summation is over all partitions $N_1\sqcup N_2=N$ so that, under the inherited orders from $\Nbold$,  $\Nbold_1$ and $\Nbold_2$ concatenate to $\Nbold$. This process iterates to longer products as 
\begin{equation}\label{eq:decompositionformula_manyfactors}
\tri^\Nbold(\g_1\g_2\cdots \g_r)=\sum_{\Nbold_1\Nbold_2\cdots\Nbold_r=\Nbold}\tri^{\Nbold_1}(\g_1)\times \tri^{\Nbold_2}(\g_2)\times\cdots \times \tri^{\Nbold_r}(\g_r).
\end{equation}
Finally, for a general partition-with-orders $\uNbold=(\Nbold_1,\ldots, \Nbold_r)$, we will write $\tri^\uNbold(\g_1\times \cdots\times \g_r)=\tri^{\Nbold_1}(\g_1)\times \cdots\times \tri^{\Nbold_r}$.

\subsection{Surfaces}
Pick a $1$-skeleton $W_{2g}$ of $\Sigma_g$ with a single $0$-cell the point $*$, and $2g$ $1$-cells $\alpha_1,\alpha_{-1},\ldots, \alpha_{g},\alpha_{-g}$ which, viewed as loops, form a basis of the free group $\pi_1=\pi_1(W_{2g},*)$ of rank $2g$. The complement $U_{2g}=W_{2g}\setminus *$ is then a disjoint union of $2g$ open intervals. 

Assume that $\Sigma_g$ is obtained by attaching a single $2$-cell along the word $\zeta\in \pi_1$ given as the product of commutators 
\begin{equation}
    \zeta=[\alpha_1,\alpha_{-1}]\cdots [\alpha_g,\alpha_{-g}].
\end{equation}
Then $\pi=\pi_1/\langle\langle \zeta\rangle \rangle$ and inherits the generators $\alpha_{\pm 1}, \ldots, \alpha_{\pm g}$. 
 
Let $D$ now be an embedded disc in $\Sigma_g$ intersecting $W_{2g}$ at $*$ on its boundary. The complement $\Sigma_{g,1}$ of the interior of $D$ has also parametrised by the loop $\zeta$.  Denote by $\Gamma_{g,1}$ the mapping class group of $\Sigma_{g,1}$.

\subsubsection{The groups $\Hcal_n(\Sigma_{g,*})$ and $\Hcal_n(U_{2g})$}\label{subsec:Hcalofsurfaces}
In the discussion of the Section \ref{subsec:recollections}, we can take $U=\Sigma_{g,*}, U_{2g}$ and $X=\Sigma_g,W_{2g}$, respectively. The inclusion $\iota:(W_{2g},*)\hookrightarrow (\Sigma_g,*)$ gives $\Hcal_N(\iota):\Hcal_N(U_{2g})\to \Hcal_N(\Sigma_{g,*})$. The group $\Hcal_N(U_{2g})$ was completely determined by Moriyama \cite{Moriyama}, see Theorem \ref{thm:Moriyama} below. 
\begin{theorem}[Theorem 2.10 \cite{LooijengaStavrou25}]\label{thm:kernelfromLS25}
    The map $\Hcal_N(\iota)$ is a surjection with kernel $$\Kcal_N(\iota)=\sum_{\Ibold, J: N=I\sqcup J, |I|\ge 2}\tri^{\Ibold}(\zeta)\times \Hcal_J(U_{2g}).$$
\end{theorem}

\subsubsection{Mapping class groups}
The inclusion $\Sigma_{g,1}\to \Sigma_{g,*}$ induces a surjection of mapping class groups $\Gamma_{g,1}\twoheadrightarrow \Gamma_{g,*}$ with kernel $\ZZ$ generated by the Dehn twist along the boundary. This allows us to view any $\Gamma_{g,*}$-representation as a $\Gamma_{g,1}$-representation as well.
The action of both groups on $H_\ZZ=H_1(\Sigma_{g,1})\cong H_1(\Sigma_{g,*})$ factors through the symplectic group $\Sp_{2g}(\ZZ)$. 

\begin{proposition}[Corollary 3.9 \cite{LooijengaStavrou25}]
    For $\Gamma=\Gamma_{g,1}$ or $\Gamma_{g,*}$, the kernel of the $\Gamma$-representation $\Ical^\cfg_n$ is $J^\cfg(n)$.
\end{proposition}

\subsubsection{Algebraic representations}
Here and onwards, let $\Gamma$ be either $\Gamma_{g,*}$ or $\Gamma_{g,1}$. The $\Gamma$-representation $H=H_\ZZ\otimes \QQ$, factors through $\Sp_{2g}(\ZZ)$, and all so-called \textit{algebraic} representations of the symplectic group are direct sums of subquotients of $H^{\otimes n}$, for $n\ge 0$. Algebraic representations reduce to direct sums of irreducibles.

\begin{definition}
    The weight $w(V)$ of an algebraic irreducible is the least $n\ge 0$ for which $V$ embeds $\Sp_{2g}(\ZZ)$-equivariantly in $H^{\otimes n}$. For each algebraic representation $U$ and $w\ge 0$, let $U^{\le w}$ be the span of all its irreducible summands of weight $\le w$.
\end{definition}

\begin{example}
    Each $H^{\otimes n}$ consists of irreducibles of weight $\le n$ and with same parity as $n$.
\end{example}

\subsubsection{Gr-algebraic representations}
We introduce a more general class of $\Gamma$-representations following a definition from \cite[\S~2.1.2]{KupersORWalgebraic}.
\begin{definition}
    We say a $\Gamma$-representation $V$ is \textit{$\gr$-algebraic} if it has a finite filtration 
    $$0= F_0(V)\subset F_1(V)\subset \cdots \subset F_p(V)=V$$
    whose associated graded $\gr^F_\pt V=\oplus_{i=1}^pF_i(V)/F_{i-1}(V)$ factors through $\Sp_{2g}(\ZZ)$ and is algebraic.
    We will say $V$ is of weight $\le d$, if $\gr^F_\pt V$ is. 
\end{definition}

\begin{remark}
    The associated graded, when viewed as an ungraded $\Sp_{2g}(\ZZ)$-representation, does not depend on the choice of filtration that exhibits the gr-algebraicity: this can be checked by repeated applications of Schur's lemma.
\end{remark}

\begin{proposition}\label{prop:gralgebraicclosedunder}
    The property of being gr-algebraic is closed under subquotients, direct sums, and tensor products. Furthermore, if $V_1$ and $V_2$ are of weight $\le d_1$ and $\le d_2$, respectively, then (a) any subquotient of $V_1$ is of weight  $\le d_1$, (b) the direct sum $V_1\oplus V_2$ is of weight $\le \max(d_1,d_2)$, and (c) the tensor product $V_1\otimes V_2$ is of weight $\le d_1+d_2$.
\end{proposition}
\begin{proof}
    An application of Lemma 2.5 of \cite{KupersORWalgebraic} and its proof.
\end{proof}

\subsubsection{Proof of Theorem \ref{athm:bottompartLambda}}
The pair $(\Sigma_{g,1},*)$ deformation retracts to $(W_{2g},*)$ giving a natural action of  $\Gamma_{g,1}$ on $\pi_1$ and on $\Hcal_n(U_{2g})$, under which the map $\Hcal_n(\iota):\Hcal_N(U_{2g})\to \Hcal_N(\Sigma_{g,*})$ from Section \ref{subsec:Hcalofsurfaces} is $\Gamma_{g,1}$-equivariant.

\begin{example}\label{example:truncatedgroupringgralgebraic}
    A theorem of Fox \cite{Fox} expresses the associated graded of $\QQ\pi_1$ by the augmentation filtration as the free tensor algebra $T[H]$, so $\QQ\pi_1/\Ical_1^{n+1}$ and $\Ical_1/\Ical_1^{n+1}$ are gr-algebraic of degree $\le n$. A similar result of Labute \cite{Labute} implies $\QQ\pi/\Ical^{n+1}$ and $\Ical/\Ical^{n+1}$ are gr-algebraic of weight $\le n$.
\end{example} 

\begin{proposition}\label{prop:Moriyamagralgebraic}
    The $\Gamma_{g,1}$-representation $\Hcal_N(U_{2g})$ is gr-algebraic of weight $\le n$. 
\end{proposition}
\begin{proof}
    By Theorem 1.5 of \cite{LooijengaStavrou25}, $\Hcal_N(U_{2g})$ is spanned by the images of the maps $\tri^{\uNbold}_{U_{2g}}$ for $\uNbold=(\Nbold_1,\ldots, \Nbold_r)$ all partitions-with-orders of the set $N$. These maps are equivariant under the action of self-homotopies of the pair $(W_{2g},*)$ and thus of $\Gamma_{g,1}$. The domain of each $\tri^{\uNbold}_{U_{2g}}$ is $\Ical_1/\Ical_1^{n_1}\otimes \cdots \Ical_1/\Ical_1^{n_r}$ where $n_i=|N_i|$ and $n_1+\cdots+n_r=n$. Now, each tensor factor is gr-algebraic of weight $\le n_i$ by Example \ref{example:truncatedgroupringgralgebraic}, so by Proposition \ref{prop:gralgebraicclosedunder} the domain of $\tri^\uNbold$ is gr-algebraic of weight $\le n_1+\ldots +n_r=n$. By the same proposition, so is the span of their images $\Hcal_N(U_{2g})$.
\end{proof}

There is a $\Gamma$-invariant element $\mu\in \Lambda^2H\subset H^{\otimes 2}$ corresponding to the intersection pairing of $H$. If $a_1,a_{-1},\ldots, a_g,a_{-g}$ are the images in $H$ of the generators of $\pi_1$ and $\pi$, it has the explicit form
\begin{equation}
    \mu=\sum_{i=1}^ga_i\otimes a_{-i}-a_{-i}\otimes a_i=\sum_{1\le \pm i\le g} \sign(i)a_i\otimes a_{-i}\in H^{\otimes 2}.
\end{equation}
\noindent Then for $1\le k<l\le n$, there is an \textit{insertion map} $\mu_{k,l}:H^{\otimes n-2}\to H^{\otimes n}, v\mapsto \mu\otimes_{k,l}v$ that inserts $\mu$ in the $k,l$ tensor slots. We note that the consecutive insertions maps $\mu_{i,i+1}$ appear in the context of $\gr^\Ical_\pt\QQ\pi$: an equivalent formulation of Labute's result is that $\Ical^n/\Ical^{n+1}\cong H^{\otimes n}/\sum_{1\le i\le n-1}\im\mu_{i,i+1}$.  More generally, these maps play a key role in symplectic representation theory.
\begin{proposition}\label{prop:mukl_weightk_Hotimesn}
    The images $\im \mu_{k,l}$ for $1\le k<l\le n$ span the weight $\le n-2$ part of $H^{\otimes n}$.
\end{proposition}
This follows directly from \cite{Procesi} and allows us to prove Theorem \ref{athm:bottompartLambda}. 
\begin{proof}[Proof of Theorem \ref{athm:bottompartLambda}]
In Corollary 3.5 of \cite{LooijengaStavrou25}, we proved that the images of all maps $\mu_{i,j}$ lie in the kernel of the composition of $p:H^{\otimes n}\twoheadrightarrow \Ical^n/\Ical^{n+1}$ with $\gr^\Ical_n\tri^n$, so the weight $\le n-2$ part of $\Ical^n/\Ical^{n+1}$ does lie in the kernel of $\gr^\Ical_n\tri^n$. 
To conclude the proof it suffices to show that this kernel is of weight $\le n-2$. 

The inclusion $\iota: W_{2g}\to \Sigma_g$ gives the commuting square of $\Gamma_{g,1}$-equivariant maps
\begin{equation}
    \begin{tikzcd}
        (\Ical_1)^n/(\Ical_1)^{n+1}\cong H^{\otimes n}\rar["\gr^{\Ical_1}_n\tri_1^n", hook]\dar[two heads] & \Hcal_n(U_{2g})\dar[two heads,"\Hcal_n(\iota)"]\\
        \Ical^n/\Ical^{n+1}\rar["\gr^\Ical_n\tri^n"] & \Hcal_n(\Sigma_{g,*}),
    \end{tikzcd}
\end{equation}
where the top map is injective by Moriyama \cite{Moriyama}. The desired kernel of $\gr^\Ical_n\tri^n$ is then image under the leftmost map of $(\gr^{\Ical_1}_n\tri_1^n)^{-1}(\ker(\Hcal_n(\iota)))$ and so has weight bounded by the same bound as $\ker(\Hcal_n(\iota))$. This kernel was described in Theorem \ref{thm:kernelfromLS25} to be spanned by the subspaces $\tri_1^\Ibold(\zeta)\times \Hcal_{[n]\setminus I}(U_{2g})$ for all subsets $I\subseteq [n]$ of size $\ge 2$ equipped with an order. Under the action of $\Gamma_{g,1}$, the element $\tri_1^\Ibold(\zeta)$ is invariant, the operation $\times$ is equivariant and $\Hcal_{[n]\setminus I}(U_{2g})$ is, by Proposition \ref{prop:Moriyamagralgebraic}, gr-algebraic of weight $\le |[n]\setminus I|\le n-2$. It follows that $\ker(\Hcal_n(\iota))$ is gr-algebraic of weight $\le n-2$ and thus $\ker\gr^\Ical_n\tri^n$ is also  gr-algebraic (and \textit{a priori} algebraic) of weight $\le n-2$ as desired.
\end{proof}

\section{The module $\Ical^\cfg_n$}\label{sec:keroftri^n}
In this section, we determine the kernel of the map $\tri^n$, reducing Theorem \ref{athm:groupring} to Theorem \ref{thm:independence_r+s=n} which we will then treat in its dedicated Section \ref{sec:geometricconstructions}. To do so, we first find $\QQ[\Gamma]$-generators for $\QQ\pi/\Ical^{n+1}$, where $\Gamma$ is either $\Gamma_{g,1}$ or $\Gamma_{g,*}$, by lifting $\QQ[\Sp_{2g}(\ZZ)]$-generators of its associated graded imported from symplectic representation theory. Then we determine which of these generators are annihilated by $\tri^n$, and argue that the rest remain $\QQ[\Gamma]$-linearly independent. 

\subsection{A Magnus correspondence}
The isomorphism $\Ical_1^{n}/\Ical_1^{n+1}\cong H^{\otimes n}$ of Fox is given by the correspondence
\begin{equation}\label{eq:correspondenceMagnus}
    (\g_1-1)\cdots (\g_n-1)\mapsto [\g_1]\otimes \cdots \otimes [\g_n]
\end{equation}
for $\g_1,\ldots, \g_n\in \pi_1$ and where $[\g]\in H$ is the abelianisation of $\g\in \pi_1$. 
Conversely, our chosen basis for $\pi_1$ produces a lift \eqref{eq:correspondenceMagnus} given on monomials by 
\begin{equation}\label{eq:MagnusSplitting}
    \vfrak=a_{i_1}\otimes \cdots \otimes a_{i_n}\in H^{\otimes n}\mapsto \widetilde{\vfrak}= (\alpha_{i_1}-1)\cdots(\alpha_{i_n}-1)\in \QQ\pi_1,
\end{equation}
for $i_1,\ldots, i_n\in \{\pm 1, \ldots, \pm g\}$. The linearly extended map $\widetilde{\cdot}:T[H]\to \QQ\pi_1$ is easily seen to be a ring homomorphism but not $\Gamma_{g,1}$-equivariant. 
\begin{remark}
    Since $\tri^{0}(\g)=1$ for all $\g\in \pi_1$ and $\tri^{i}(1)=0$ for all $i\ge 1$, it follows that $\tri^{0}$ vanishes on $\g-1$ and $\tri^{i}(\g-1)=\tri^{i}(\g)$ for $i\ge 1$. The decomposition formula \eqref{eq:decompositionformula_manyfactors} then takes on the particular form on $\tri^n(\widetilde{\vfrak})$ where $\vfrak=a_{i_1}\otimes \cdots \otimes a_{i_k}\in H^{\otimes k}$:
    \begin{equation}\label{eq:tri^ngeneraltilde}
    \tri^n(\widetilde{\vfrak})=\sum_{\Nbold_1\cdots\Nbold_k=\textbf{[n]}} \tri^{\Nbold_1}(\alpha_{i_1})\times\cdots \times \tri^{\Nbold_k}(\alpha_{i_k})
    \end{equation}
    where the summation this time only runs over all \textit{non-empty} partitions-with-orders $\Nbold_1,\ldots,$ $ \Nbold_k$ that concatenate to the standard order $(1, \ldots, n)$. 
\end{remark}

\begin{example}\label{ex:tri^2(mu)}
    Since $\zeta$ is a commutator, it follows from Fox \cite{Fox} that $\zeta-1\in \Ical_1^2$, and we checked in \cite{LooijengaStavrou25} that $\zeta-1+\Ical_1^3\mapsto \mu$ under correspondence \eqref{eq:correspondenceMagnus}.
    In particular, formula \eqref{eq:tri^ngeneraltilde} gives
    \begin{equation}\label{eq:tri^nmutilde}
        \tri^n(\widetilde{\mu})=\sum_{ 0<k<n}\sum_{1\le \pm i\le g}\sign(i)\tri^{(1,\ldots, k)}(\alpha_{i})\times \tri^{(k+1,\ldots, n)}(\alpha_{-i}). 
    \end{equation}
    However, $\mu\neq \zeta-1\in \QQ\pi_1$. In particular, $\tri_1^k(\widetilde{\mu})\neq\tri_1^k(\zeta)\in \Hcal_k(U_{2g})$ for $k\ge 3$ or, equivalently, $\tri^k(\widetilde{\mu})=0\in \Hcal_k(\Sigma_{g,*})$ for $k\ge 3$. Chasing this discrepancy between $\widetilde{\mu}$ (that comes from symplectic representation theory) and $\zeta$ (that comes from topology and group theory) is the crux of this section.
\end{example} 

\subsection{$\QQ[\Sp_{2g}(\ZZ)]$-generators}\label{sec:SPgenerators}
We consider monomials $\mfrak\in T[H]$ in  the free associative generators $a_{\pm 1},\ldots, a_{\pm g}$. Each such monomial is said to be \textit{of type $(p,q)$} if it has $p$ positive-index factors and $q$ negative-index ones; say $\mfrak$ is \textit{positive} if $q=0$. 
Classical representation theory (and this is implied by what follows) gives us that the weight-$k$ part of $H^{\otimes n}$ is generated as an $\Sp_{2g}(\ZZ)$-representation by the monomials of type $(k,n-k)$.  We will however give a more restricted set of generators.
\begin{definition}
    A \textit{chord diagram (of length $r$)} in the set $[n]$ is a pair $(\uk,\ul)$ of $r$-tuples $\uk=(k_1,\ldots, k_r)$ and $\ul=(l_1,\ldots, l_r)$ of a total of $2r$ pairwise distinct elements of $[n]$, with $k_i<l_i$ for all $i=1,\ldots, r$, and $k_1<k_2<\cdots<k_r$. Say that $(\uk,\ul)$ is \textit{non-consecutive} if all pairs $(k_i,l_i)$ are non-consecutive, that is $k_i+1<l_i$ for all $i=1,\ldots, r$, and call it \textit{consecutive} otherwise. 
\end{definition} 
For any chord diagram, there is a corresponding insertion map \begin{equation}
    \mu^{n,r}_{\uk,\ul}:H^{\otimes n-2r}\to H^{\otimes n}
\end{equation} the map inserting $r$ copies of $\mu$ in the pairs of tensor slots $(k_1,l_1), \ldots, (k_r,l_r)$, so that if $r=1$, $\mu^{n,1}_{(k,l)}$ is the map $\mu_{k,l}$ of Proposition \ref{prop:mukl_weightk_Hotimesn}. Equivalently, each $\mu^{n,r}_{\uk,\ul}$ is a composition of the map $\vfrak\in H^{\otimes n-2r}\mapsto \mu^{\otimes r}\otimes \vfrak$ with a permutation of the tensor factors by some $\sigma\in \Sfrak_n$. 

\begin{definition}\label{def:positivesubspace}
    Let $\MM^n_+$ be the set of all \textit{positive monomials} in $H^{\otimes n}$, and denote by $\HH^n _+$ their $\QQ$-span. For $1\le k\le g$, let $\MM^n_{+,k}\subset \MM^n_+$ be the subset of monomials involving only the last $k$ indices, i.e.  $\mfrak\in \MM^n_{+,k}$ is a product of $a_i$ with $i>g-k$. Denote by $\HH^n _{+,k}$ the $\QQ$-span of this set.
\end{definition}

For $0\le 2r\le n$, we define
\begin{equation*}
    \BB^{n,r}=\{\mu^{n,r}_{\uk,\ul}(\mfrak): (\uk,\ul)\text{ chord-diagrams of length } r\text{, and } \mfrak\in \MM^{n-2r}_{+, n-2r}\},
\end{equation*}
and $\BB^{n,r}_\nc$ its subset where we only consider \textit{non-consecutive} chord-diagrams. Note that $\BB^{0,0}=\{1\}$ is the set containing the empty monomial.

\begin{proposition}\label{prop:weight(n-2r)ofH^n}
    For $0\le 2r\le n$, the weight-$(n-2r)$ part of $H^{\otimes n}$ is generated over $\QQ[\Sp_{2g}(\ZZ)]$ by $\BB^{n,r}$. Furthermore, each irreducible summand of $H^{\otimes n}$ of weight $(n-2r)$ contains a non-zero $\QQ$-linear combination of such elements. More specifically, this element is the application of a linear combination of $\mu^{n,r}_{\uk,\ul}$, for various chord-diagrams $(\uk,\ul)$, on a fixed $\vfrak\in \HH^{n-2r}_{+,n-2r}$.
\end{proposition}
\begin{proof}
    All statements used here come from \S 17.3 of Fulton--Harris \cite{FultonHarris}. The weight-$(n-2r)$ part of $H^{\otimes n}$ lies in the span of the images of the repeated insertions $\mu_{\uk,\ul}^{n,r}$. More specifically, it is the image of a weight-$(n-2r)$ irreducible summand $V$ of $H^{\otimes n-2r}$  under a linear combination of maps $\mu^{n,r}_{\uk,\ul}$. Now, this irreducible $V$ has a \textit{highest weight vector}, i.e. a $\QQ[\Sp_{2g}(\ZZ)]$-generator $\vfrak$, that is a sum of permutations of a specific $a_{i_1}\otimes \cdots \otimes a_{i_{n-2r}}$ where $1=i_1\le \ldots \le i_{n-2r}$. There is a symplectic transformation $A\in \Sp_{2g}(\ZZ)$ that swaps the order $a_{\pm i}\mapsto a_{\pm (n-2r-i)}$, putting these monomials in $\MM^{n-2r}_{+,n-2r}$. 
\end{proof}

\begin{proposition}\label{prop:weight(n-2r)ofgrI_QQpi}
    For $0\le 2r\le n$, the weight-$(n-2r)$ part of $\Ical^n/\Ical^{n+1}$ is generated over $\QQ[\Sp_{2g}(\ZZ)]$ by $\BB^{n,r}_\nc$. Furthermore, each irreducible summand of $\Ical^n/\Ical^{n+1}$ of weight $(n-2r)$ contains a non-zero $\QQ$-linear combination of such elements.
\end{proposition}
\begin{proof}
    The algerba $\gr^I_\pt\QQ\pi$ is the quotient of the tensor algebra $T[H]$ by the two-sided ideal generated by $\mu$, so the surjection $H^{\otimes n} \to \Ical^n/\Ical^{n+1}$ has kernel the images of all consecutive insertions $\mu^{n,r}_{\uk,\ul}$. The conclusion follows from Proposition \ref{prop:weight(n-2r)ofH^n}.
\end{proof}

Denote by $\widetilde{\BB^{n,r}}$ and $\widetilde{\BB^{n,r}_\nc}$ the sets of lifts of ${\BB^{n,r}}$ and ${\BB^{n,r}_\nc}$, respectively, in $\QQ\pi_1$ according to \eqref{eq:MagnusSplitting}. The following is a recursive application of Propositions \ref{prop:weight(n-2r)ofH^n} and \ref{prop:weight(n-2r)ofgrI_QQpi}.

\begin{proposition}\label{prop:QQGammaGenGroupRing}
    For any $n\ge 0$, the $\QQ[\Gamma_{g,1}]$-module $\QQ\pi_1/\Ical_1^{n+1}$ (resp. $\QQ[\Gamma]$-module $\QQ\pi/\Ical^{n+1}$) is spanned by the union of the sets  $\widetilde{\BB^{s,r}}$ (resp. $\widetilde{\BB^{s,r}_\nc}$) with $0\le 2r\le s\le n$.
\end{proposition}

\subsection{Lower bound for $\ker\tri^n$}
We now determine which basis elements are for certain in the kernel of $\tri^n:\QQ\pi/\Ical^{n+1}\to \Hcal_n(\Sigma_{g,*})$. 
\begin{theorem}\label{thm:lowerboundker}
    Suppose $0\le 2r\le s\le n$. Then the lifted image of $\mu_{\uk,\ul}^{s,r}$ in $\QQ\pi/\Ical^{n+1}$ is annihilated by  $\tri^n$ whenever $s+r\ge n+1$.
\end{theorem}
\begin{proof}
    The codomain $H^{\otimes s}$ of $\mu^{s,r}_{\uk,\ul}$ is lifted in $\QQ\pi/\Ical^{n+1}$ into the span $\widetilde{H^{\otimes s}}$ of products of type $(\g_1-1)\cdots (\g_s-1)$ where $\g_1,\ldots, \g_s\in \{\alpha_{\pm1}, \ldots, \alpha_{\pm g}\}$. The decomposition formula \eqref{eq:decompositionformula_manyfactors}
    evaluates $\tri^n$ on each of these products as the sum
    $$\sum_{\Nbold_1\cdots \Nbold_s=\textbf{n}}\tri^{\Nbold_1}(\g_1-1)\times \cdots \times \tri^{\Nbold_s}(\g_s-1)$$
    over all ordered subsets $\Nbold_1,\ldots, \Nbold_s$ that concatenate to the order $1\prec \cdots \prec n$. But if any $\Nbold_i$ is empty, then $\tri^{\Nbold_i}(\g_i-1)=\tri^{0}(\g_i-1)=\aug(\g_i-1)=0$, and so the product vanishes. Then in fact, we have that $\tri^n$ restricted to $\widetilde{H^{\otimes s}}$ decomposes as $\sum \tri^{\Nbold_1,\ldots, \Nbold_s}$ over all \textit{non-empty} partitions-with-orders  $\Nbold_1,\ldots, \Nbold_s$ that concatenate to the order $1\prec \cdots \prec n$.

    Given a monomial $\mfrak\in H^{\otimes s-2r}$, we will show that $\tri^{\Nbold_1,\ldots, \Nbold_s}(\widetilde{\mu^{s,r}_{\uk,\ul}(\mfrak)})$ vanishes for all non-empty partitions-with-orders $\Nbold_1,\ldots, \Nbold_s$ which, from above, implies the vanishing of $\tri^n(\widetilde{\mu^{s,r}_{\uk,\ul}(\mfrak)})$. Now fix such an $\Nbold_1,\ldots, \Nbold_s$.   
    The assumption $n\le s+r-1$ and the pigeonhole principle imply the existence of an $i\in \{1,\ldots, r\}$ for which both $N_{k_i}, N_{l_i}$ are singletons; say their elements are $p_1, p_2\in [s]$, respectively. We then have the factorisation $\mu^{s,r}_{\uk,\ul}(\mfrak)=\mu\otimes_{p_1,p_2} v$ for some $v\in H^{\otimes s-2}$ and, as a result, $$\tri^{\Nbold_1, \ldots, \Nbold_s}(\widetilde{\mu^{s,r}_{\uk,\ul}(\mfrak)})=\tri^{(p_1,p_2)}(\widetilde{\mu})\times \tri^{\Nbold_1,\ldots,\widehat{\Nbold_{k_i}},\ldots, \widehat{\Nbold_{l_i}}, \ldots \Nbold_s}(\widetilde{v}).$$ But here the first factor vanishes from Example \ref{ex:tri^2(mu)}, and so the product also does, as desired.
\end{proof}

In the special case $r=1, s=n$, this recovers Corollary 3.5 of \cite{LooijengaStavrou25}. The following is now immediate.
\begin{corollary}\label{cor:vanishingsr>N}
   We have the vanishing $\tri^n(\widetilde{\bfrak})=0\in \Hcal_n(\Sigma_{g,*})$ for any $\bfrak\in \BB^{s,r}_\nc$ with  $s+r\ge n+1$.
\end{corollary}

\subsection{Upper bound for $\ker\tri^n$}
We will show that our lower bound $\ker(\tri^n)$ is sharp, if the genus is large, in the sense that all elements of $\cup_{s+r\le n}\BB^{s,r}_\nc$ that were not annihilated by $\tri^n$ according to Corollary \ref{cor:vanishingsr>N} are actually mapped to linearly independent elements of $\Hcal_n(\Sigma_{g,*})$. The following theorem is the key ingredient to proving sharpness. We state the theorem here but will prove it separately in Section \ref{sec:geometricconstructions} with geometric methods. 

\begin{theorem}\label{thm:independence_r+s=n}
    If $0\le n \le g$, the set $$\{\tri^n(\widetilde{\bfrak})\in \Hcal_n(\Sigma_{g,*}):\bfrak\in \cup_{s+r=n}\BB^{s,r}_{\nc} \}$$ is $\QQ$-linearly independent. Furthermore, for fixed $0\le s\le n$ (note that this fixes $r=(n-s)/2$), so is the subset
    $$\{\tri^n(\widetilde{\bfrak})+\tri^n(\Ical^{s+1})\in \Hcal_n(\Sigma_{g,*})/\tri^n(\Ical^{s+1}):\bfrak\in \BB^{s,r}_{\nc}\}.$$
\end{theorem}

\begin{corollary}\label{cor:independence_r+s<=n}
    If $0\le n \le g$, the set  
    $$\{\tri^n(\widetilde{\bfrak})\in \Hcal_n(\Sigma_{g,*}):\bfrak\in \cup_{s+r\le n}\BB^{s,r}_{\nc} \}$$ is $\QQ$-linearly independent. Furthermore, for fixed $0\le s\le n$, so is the set
    $$\{\tri^n(\widetilde{\bfrak})+\tri^n(\Ical^{s+1})\in \Hcal_n(\Sigma_{g,*})/\tri^n(\Ical^{s+1}):\bfrak\in \cup_{s+r\le n} \BB^{s,r}_{\nc}\}.$$
\end{corollary}
\begin{proof}
    Proceed inductively on $n$, with the case $n=0$ being obvious. In \cite{LooijengaStavrou25}, we found a degeneracy map $\partial:\Hcal_n(\Sigma_{g,*})\to \Hcal_{n-1}(\Sigma_{g,*})$ that commutes with $\tri^n$ and $\tri^{n-1}$, and in particular makes the following diagram commute
    \begin{equation}
        \begin{tikzcd} \bigoplus_{s+r=n}\QQ\widetilde{\BB^{s,r}}\rar\dar[hook,"\tri^n_U"]&\bigoplus_{s+r\le n}\QQ\widetilde{\BB^{s,r}} \rar\dar["\tri^n"]& \bigoplus_{s+r\le n-1}\QQ\widetilde{\BB^{s,r}} \dar["\tri^{n-1}", hook]\\
             \Hcal_n(\Sigma_{g,*})\rar["="]& \Hcal_n(\Sigma_{g,*})\rar["\partial"] & \Hcal_{n-1}(\Sigma_{g,*}).
        \end{tikzcd}
    \end{equation}
    The top row is clearly exact, leftmost vertical map is injective by Theorem \ref{thm:independence_r+s=n} and the rightmost vertical map is injective by induction. It follows that the middle map is injective as well.

    The same inductive proof works for the ``furthermore'' assertion as $\partial$ descends to a map $\Hcal_n(\Sigma_{g,*})/\tri^n(\Ical^{s+1})\to\Hcal_{n-1}(\Sigma_{g,*})/\tri^{n-1}(\Ical^{s+1})$.
\end{proof}

The image $\Ical^\cfg_n=\im\tri^n$ is naturally filtered by $\tri^n(\Ical)\supset \tri^n(\Ical^2)\supset \cdots$ giving the (tautologically) surjective associated graded map $$\gr^\Ical_\pt:\gr^\Ical_\pt\QQ\pi/\Ical^{n+1}\to \gr^\Ical_\pt\Ical^\cfg_n,$$
where we can replace the domain by simply $\gr^\Ical_\pt\QQ\pi$.
The graded kernel $\ker\gr^\Ical_\pt\tri^n$ coincides (also tautologically) with the associated graded of $\ker\tri^n$ by the filtration $\ker\tri^n\cap \Ical\supset \ker\tri^n\cap \Ical^2\supset \cdots$. We now completely describe the kernel and image of $\gr^\Ical_\pt\tri^n$ provided $g$ is large.

\begin{theorem}\label{thm:assgradedIcalcfg} For $1\le k\le n$, the kernel of $\gr^\Ical_k\tri^n$ contains the subrepresentation $(\gr^\Ical_k\QQ\pi)^{\le 3k-2(n+1)}$, and $\gr^\Ical_k\Ical^\cfg_n$ is a quotient of
    \begin{equation}
        \gr^\Ical_k\Ical^\cfg_n\cong(\gr^\Ical_k\QQ\pi)/(\gr^\Ical_k\QQ\pi)^{\le 3k-2(n+1)}.
    \end{equation}
    
    If furthermore, $g\ge n$, both of these statements are equalities.
    In particular, if $k<2(n+1)/3$, then $\gr^\Ical_k$ is injective and  $\gr^\Ical_k\Ical^\cfg_n\cong\gr^\Ical_k\QQ\pi$. 
\end{theorem}
\begin{proof}
    It is easy to see that $(\gr^\Ical_k\QQ\pi)^{\le 3k-2(n+1)}$ lies in $\ker\gr^\Ical_k\tri^n$. From Proposition \ref{prop:weight(n-2r)ofgrI_QQpi}, $(\gr^\Ical_k\QQ\pi)^{\le 3k-2(n+1)}$ is generated by the images of the maps $\mu^{k,r}_{\uk,\ul}$ for $k-2r\le 3k-2(n+1)$. The latter inequality is equivalent to the condition $k+r\ge n+1$ under which, as we found in Theorem \ref{thm:lowerboundker}, $\tri^n$ annihilates the lifts of $\im\mu^{k,r}_{\uk,\ul}$. Then $\gr^\Ical_k\tri^n$ also does. This proves the first statement.

    Now assume $g\ge n$ and $V\subset \gr^\Ical_k\QQ\pi$ is an irreducible summand of $\ker \gr^\Ical_k\tri^n$ of weight  $k-2r>3k-2(n+1)$. By Proposition \ref{prop:weight(n-2r)ofgrI_QQpi}, there is a non-zero element $\vfrak\in V$ in the $\QQ$-span of $\BB^{k,r}_\nc$. The inequality is equivalent to $k+r\le n$, and the ``furthermore'' part of Corollary \ref{cor:independence_r+s<=n}, gives that $\tri^n(\widetilde{\vfrak})\not\equiv 0 \pmod{\tri^n(\Ical^{k+1})}$, contradicting that $\gr^\Ical_k\tri^n(\vfrak)=0$. This makes the inclusion of the first paragraph sharp.
\end{proof}

\section{A zoo of submanifolds}\label{sec:geometricconstructions}
Our current task is to prove Theorem \ref{thm:independence_r+s=n}. Henceforth we assume $g\ge n$ and we must detect the linear independence of the elements $\tri^n(\widetilde{\bfrak})\in \Hcal_n(\Sigma_{g,*})$ for $\bfrak\in \cup_{s+r=n}\BB^{s,r}$. To do so, we will employ the 
intersection perfect pairing $$\langle \hspace{4pt}, \hspace{2pt}\rangle:H^{BM}_{n}(\conf_{n}(\Sigma_{g,*}))\otimes H_{n}(\conf_{n}(\Sigma_{g,*}))\to \QQ,$$
where $H^{BM}_{n}(\conf_{n}(\Sigma_{g,*}))$ is naturally isomorphic to $\Hcal_{n}(\Sigma_{g,*})$, and construct dual elements $[\sEcal_\bfrak]\in H_n(\conf_n(\Sigma_{g,*}))$ for each $\tri^n(\widetilde{\bfrak})$ so that
\begin{equation}\label{eq:sEcalPerfectPairing}
    \langle \tri^n(\widetilde{\bfrak}), [\sEcal_{\bfrak'}]\rangle=\begin{cases}\pm 1,\text{ if } \bfrak=\bfrak',\\
0\text{ otherwise.}
\end{cases}
\end{equation}
Each $[\sEcal_\bfrak]$ will be the homology class of an $n$-dimensional, closed, oriented submanifold of $\conf_n(\Sigma_{g,*})$ specifically designed for \eqref{eq:sEcalPerfectPairing} to hold. The key ingredient for $\sEcal_\bfrak$ is the submanifold $\Ecal\subset \conf_3(\Sigma_{g,*})$ of Subsection \ref{sec:Ecal}.  The proof of \eqref{eq:sEcalPerfectPairing} is given in Theorem \ref{thm:sEcalPerfectPairing}. The pairing $\langle \hspace{4pt}, \hspace{2pt}\rangle$  will be evaluated on pairs $([M],[N])$ of submanifold classes of oriented, $n$-dimensional submanifold where $M$ is properly embedded and $N$ is closed, and will be computed as a signed, transversal intersection.

\subsection{A chartography of the surface}
We pick explicit models $\Xcal$ and $\Scal$ of $\Sigma_{g}$ and $\Sigma_{g,*}$, respectively, using the upper half plane $\{z\in \CC:\Im(z)\ge 0\}$. Take the rectangle $[0,4g]\times [0,h]\subset \CC$, for some large $h$, and partition the side on the real axis into $4g$ segments of equal lengths and identified in pairs according to Figure \ref{fig:cartographyofsurface} by isometries; then collapse the other three sides and the points $(0,i)$, for $i\in \ZZ$, $0\le i\le 4g$, into the unique point $p$. This forms the based quotient $(\Xcal,*)$;  let $\Scal=\Xcal\setminus *$ be the complement. The orientation of $\CC$ canonically orients $\Xcal$ and $\Scal$, and the Euclidean metric gives a metric on $\Scal$. 

Let $\alpha_1,\alpha_{-1},\ldots, \alpha_g,\alpha_{-g}$ be the based loops in $(\Xcal,*)$ shown in Figure \ref{fig:cartographyofsurface} parametrised by constant speed. These loops define a $1$-skeleton of $\Scal$ which we identify with the wedge $W_{2g}$.
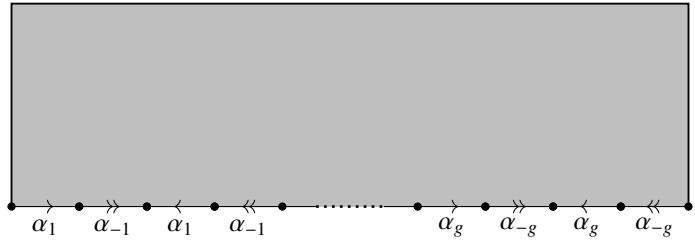
\begin{figure}
  \centering 
   \begin{tikzpicture}[xscale=0.9,yscale=0.9]

    \filldraw[color=lightgray] (0,0) rectangle ++(10,3);
    \draw[line width=0.7pt] (0,0) to (0,3) to (10,3) to (10,0);
    \filldraw (0,0) circle (1.5pt);
    \draw[postaction={decoration={markings, mark=at position 0.6 with {\arrow{>}}}, decorate}](0.05,0) to ++(0.95,0);
    \draw (0.5,-0.3) node {\footnotesize $\alpha_1$};
    \filldraw (1,0) circle (1.5pt);
    \begin{scope}[shift={(1,0)}]
        \filldraw (0,0) circle (1.5pt);
    \draw[postaction={decoration={markings, mark=at position 0.6 with {\arrow{>>}}}, decorate}](0,0) to ++(1,0);
    \draw (0.5,-0.3) node {\footnotesize $\alpha_{-1}$};
    \filldraw (1,0) circle (1.5pt);
    \end{scope}
    \begin{scope}[shift={(2,0)}]
        \filldraw (0,0) circle (1.5pt);
    \draw[postaction={decoration={markings, mark=at position 0.5 with {\arrow{<}}}, decorate}](0,0) to ++(1,0);
    \draw (0.5,-0.3) node {\footnotesize $\alpha_1$};
    \filldraw (1,0) circle (1.5pt);
    \begin{scope}[shift={(1,0)}]
        \filldraw (0,0) circle (1.5pt);
    \draw[postaction={decoration={markings, mark=at position 0.6 with {\arrow{<<}}}, decorate}](0,0) to ++(1,0);
    \draw (0.5,-0.3) node {\footnotesize $\alpha_{-1}$};
    \filldraw (1,0) circle (1.5pt);
    \end{scope}
    \end{scope}
    
    \draw (4,0) to ++(0.5,0);
    \draw[dotted, line width=1pt] (4.5,0) to (5.5,0); 
    \draw (5.5,0) to ++(0.5,0);
    
    \begin{scope}[shift={(6,0)}]
        \filldraw (0,0) circle (1.5pt);
    \draw[postaction={decoration={markings, mark=at position 0.6 with {\arrow{>}}}, decorate}](0,0) to ++(1,0);
    \draw (0.5,-0.3) node {\footnotesize $\alpha_g$};
    \filldraw (1,0) circle (1.5pt);
    \begin{scope}[shift={(1,0)}]
        \filldraw (0,0) circle (1.5pt);
    \draw[postaction={decoration={markings, mark=at position 0.6 with {\arrow{>>}}}, decorate}](0,0) to ++(1,0);
    \draw (0.5,-0.3) node {\footnotesize $\alpha_{-g}$};
    \filldraw (1,0) circle (1.5pt);
    \end{scope}
    \begin{scope}[shift={(2,0)}]
        \filldraw (0,0) circle (1.5pt);
    \draw[postaction={decoration={markings, mark=at position 0.5 with {\arrow{<}}}, decorate}](0,0) to ++(1,0);
    \draw (0.5,-0.3) node {\footnotesize $\alpha_g$};
    \filldraw (1,0) circle (1.5pt);
    \begin{scope}[shift={(1,0)}]
        \filldraw (0,0) circle (1.5pt);
    \draw[postaction={decoration={markings, mark=at position 0.6 with {\arrow{<<}}}, decorate}](0,0) to ++(0.95,0);
    \draw (0.5,-0.3) node {\footnotesize $\alpha_{-g}$};
    \filldraw (1,0) circle (1.5pt);
    \end{scope}
    \end{scope}
    \end{scope}

   \end{tikzpicture}
   \caption{\footnotesize The closed surface $\Xcal$. It is obtained by identifying the pairs of intervals labelled $\alpha_{\pm i}$ together by preserving the diretions of the arrows; the top three sides of the rectangle as well as the $4g+1$ labelled points on the bottom side are all collapsed to the basepoint $p$.}
   \label{fig:cartographyofsurface}
\end{figure}

\begin{figure}
  \centering 
   \begin{tikzpicture}[xscale=1,yscale=1]

    \filldraw[color=lightgray] (0,0) rectangle ++(9,3);
    \draw[dotted, line width=1pt] (0,0) to (0,3) to (9,3) to (9,0);
    \draw (0,0) to (0.45,0);

    \begin{scope}[shift={(0.5,0)}]

    \draw (0,0) circle (1.5pt);
    \draw[postaction={decoration={markings, mark=at position 0.8 with {\arrow{>}}}, decorate}](0.05,0) to ++(0.9,0);
    \draw (0.5,-0.3) node {\footnotesize $\alpha_i$};
    \draw (1,0) circle (1.5pt);
    \begin{scope}[shift={(1,0)}]
        \draw (0,0) circle (1.5pt);
    \draw[postaction={decoration={markings, mark=at position 0.8 with {\arrow{>}}}, decorate}](0.05,0) to ++(0.9,0);
    \draw (0.5,-0.3) node {\footnotesize $\alpha_{-i}$};
    \draw (1,0) circle (1.5pt);
    \end{scope}
    \begin{scope}[shift={(2,0)}]
        \draw (0,0) circle (1.5pt);
    \draw[postaction={decoration={markings, mark=at position 0.3 with {\arrow{<}}}, decorate}](0.05,0) to ++(0.9,0);
    \draw (0.5,-0.3) node {\footnotesize $\alpha_i$};
    \draw (1,0) circle (1.5pt);
    \begin{scope}[shift={(1,0)}]
        \draw (0,0) circle (1.5pt);
    \draw[postaction={decoration={markings, mark=at position 0.3 with {\arrow{<}}}, decorate}](0.05,0) to ++(0.9,0);
    \draw (0.5,-0.3) node {\footnotesize $\alpha_{-i}$};
    \draw (1,0) circle (1.5pt);
    \end{scope}
    \end{scope}

    \draw[domain=0:180, postaction={decoration={markings, mark=at position 0.6 with {\arrow{<}}}, decorate}] plot ({1.5+cos(\x)}, {sin(\x)});
    \draw (0.95,1.1) node {\footnotesize $\beta_i$};
    \draw[domain=0:180, postaction={decoration={markings, mark=at position 0.4 with {\arrow{<}}}, decorate}] plot ({2.5+cos(\x)}, {sin(\x)});
    \draw (3.2,1.1) node {\footnotesize $\beta_{-i}$};
    \filldraw (2,{sqrt(3)/2}) circle (1pt);
    \draw (2,{sqrt(3)/2-0.3}) node {\footnotesize $p_i$};

    \begin{scope}[shift={(4,0)}]
    
        \draw (0,0) circle (1.5pt);
    \draw[postaction={decoration={markings, mark=at position 0.8 with {\arrow{>}}}, decorate}](0.05,0) to ++(0.9,0);
    \draw (0.5,-0.3) node {\footnotesize $\alpha_{i+1}$};
    \draw (1,0) circle (1.5pt);
    
    \begin{scope}[shift={(1,0)}]
        \draw (0,0) circle (1.5pt);
    \draw[postaction={decoration={markings, mark=at position 0.8 with {\arrow{>}}}, decorate}](0.05,0) to ++(0.9,0);
    \draw (0.5,-0.3) node {\footnotesize $\alpha_{-i-1}$};
    \draw (1,0) circle (1.5pt);
    \end{scope}
    
    \begin{scope}[shift={(2,0)}]
    
        \draw (0,0) circle (1.5pt);
    \draw[postaction={decoration={markings, mark=at position 0.3 with {\arrow{<}}}, decorate}](0.05,0) to ++(0.9,0);
    \draw (0.5,-0.3) node {\footnotesize $\alpha_{i+1}$};
    \draw (1,0) circle (1.5pt);
    
    \begin{scope}[shift={(1,0)}]
    
        \draw (0,0) circle (1.5pt);
    \draw[postaction={decoration={markings, mark=at position 0.3 with {\arrow{<}}}, decorate}](0.05,0) to ++(0.9,0);
    \draw (0.5,-0.3) node {\footnotesize $\alpha_{-i-1}$};
    \draw (1,0) circle (1.5pt);
    
    \end{scope}
    \end{scope}
    \draw[domain=0:180, postaction={decoration={markings, mark=at position 0.6 with {\arrow{<}}}, decorate}] plot ({1.5+cos(\x)}, {sin(\x)});
    \draw (0.9,1.1) node {\footnotesize $\beta_{i+1}$};
    \draw[domain=0:180, postaction={decoration={markings, mark=at position 0.4 with {\arrow{<}}}, decorate}] plot ({2.5+cos(\x)}, {sin(\x)});
    \draw (3.3,1.1) node {\footnotesize $\beta_{-i-1}$};
    \filldraw (2,{sqrt(3)/2}) circle (1pt);
    \draw (2,{sqrt(3)/2-0.4}) node {\footnotesize $p_{i+1}$};
    \end{scope}
    
    \end{scope}
    
    \draw (8.55,0) to (9,0);

    \draw[domain={atan(sqrt(3)/4)}:{180-atan(sqrt(3)/4)}, postaction={decoration={markings, mark=at position 0.4 with {\arrow{<}}}, decorate}] plot ({4.5+sqrt(19/4)*cos(\x)}, {sqrt(19/4)*sin(\x)});
    \draw (4.8,2.35) node {\footnotesize $\epsilon_{i,i+1}$};
    
    \draw[domain={acos(0.5/sqrt(19/4))}:{180-atan(sqrt(3)/4)}, postaction={decoration={markings, mark=at position 0.4 with {\arrow{<}}}, decorate}] plot ({8.5+sqrt(19/4)*cos(\x)}, {sqrt(19/4)*sin(\x)});
    \draw (8,2.35) node {\footnotesize $\epsilon_{i+1,i+2}$};

    \draw[domain={atan(sqrt(3)/4)}:{180-acos(0.5/sqrt(19/4))}, postaction={decoration={markings, mark=at position 0.6 with {\arrow{<}}}, decorate}] plot ({0.5+sqrt(19/4)*cos(\x)}, {sqrt(19/4)*sin(\x)});
    \draw (1,2.35) node {\footnotesize $\epsilon_{i+1,i+2}$};

   \end{tikzpicture}
   \caption{\footnotesize The dual curves $\beta_{\pm i}$, the intersection points $p_i$, and the paths $\epsilon_{i,i+1}$ in the open surface $U$. (Here only a part of $U$ is depicted.)}
   \label{fig:dualcurves}
\end{figure}
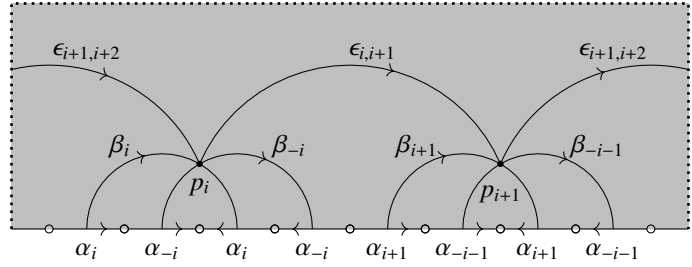

\subsubsection{Curve intersections and dual curves}
The transversal intersection between a pair of curves $\gamma,\delta$ in $\Scal$ or $\Xcal$ is \textit{positive} if the pair of tangent vectors $(\gamma',\delta')$ at the intersection is oriented as the pair $(1,i)\in \CC^2$. A simple closed curve $\gamma$ in $\Scal$ that intersects the $U_{2g}$ transversally is specified up to isotopy by its \textit{intersection pattern} with the $\alpha_{i}$: a sequence of intersections specified by the $\alpha_i$ and a sign. In particular, we say that $\gamma$ is \textit{dual} to $\alpha_i$ if $\gamma,\alpha_i$ intersect once transversally and positively, and $\gamma$ is disjoint from all other generating loops. Our favourite dual curve to $\alpha_i$ is the hyperbolic line $\beta_i$ between the midpoints of $\alpha_i$ and $\alpha_{i}^{-1}$ as in Figure \ref{fig:dualcurves}; so then $\beta_i$ and $\beta_{-i}$ intersect transversally once, positively at the the point $p_i$. We fix the hyperbolic arc $\epsilon_{i,j}$ from $p_i$ to $p_j$. An open thickening of the graph in $\Scal$ comprised of $$\beta_i,\beta_{-i},\epsilon_{i,i+1},\beta_{i+1},\beta_{-(i+1)},\epsilon_{i+1,i+2},\ldots,\epsilon_{j-1,j} \beta_{j},\beta_{-j}, $$
for any $i\le j$,
gives the open subsurface $\Scal_{[i,i+1,\ldots,j]}$; it is homeomorphic to the complement of a point in closed orientable surface of genus $j-i+1$.



\subsection{Intersections of cubes in $\conf_n(\Scal)$}\label{sec:intersectionconventions}
If $\gamma_1,\ldots, \gamma_k: (0,1)\to \Scal$ are pairwise disjoint simple open segments in $\Scal$, then we can take the open $n$-cube $$\gamma_1\times \cdots \times \gamma_n:(0,1)^k\to \conf_k(\Scal).$$ More generally, if $\sigma\in \Sfrak_k$ is any permutation, then we can take the cube $\gamma_1^{(\sigma(1))}\times \cdots \times \gamma_n^{(\sigma(k))}$ where the particle $\sigma(i)$ traverses $\gamma_i$. (Here and onwards, the $\times$ symbol is viewed as coming from an iteration of the partially defined product $\conf_{I}(\Scal)\times \conf_{J}(\Scal) \dashrightarrow \conf_{I\sqcup J}(\Scal)$.) 
The open submanifold $\gamma_1^{(\sigma(1))}\times \cdots \times \gamma_n^{(\sigma(k))}$ is canonically oriented using the product orientation of $(0,1)^k$. We will be interested to know the signs of intersecton of two such submanifolds and whether they intersect transversally. 

\begin{propositionconvention}\label{prop:intersectionconvention}
Suppose $\g_1,\delta_1,\ldots, \g_n,\delta_n$ are open oriented  segments in $\Scal$, and $\sigma,\tau\in \Sfrak_n$ are permutations,
such that each pair $\g_{\sigma^{-1}(i)}$, $\delta_{\tau^{-1}(i)}$ intersects transversely $k_i$-many times, with signed intersection count $\varepsilon_i$. Then the open oriented submanifolds $\g_1^{(\sigma(1))}\times \cdots \times \g_n^{(\sigma(n))}$ and $\delta_1^{(\tau(1))}\times \cdots \times \delta_n^{(\tau(n))}$ of $\Scal^n$ intersect transversely $\prod_{i=1}^n k_i$-many times with signed intersection count $$\varepsilon(\sigma,\tau)\cdot \prod_{i=1}^n \varepsilon_i,$$ where $\varepsilon(\sigma,\tau)$ is the sign of the permutation that takes the product
$\g_*^{(\sigma(1))}\times \cdots \times \g_*^{(\sigma(n))}\times \delta_*^{(\tau(1))}\times \cdots \times \delta_*^{(\tau(n))}$ into $\g_*^1\times \delta_*^1\times \cdots \times \g_*^n\times \delta_*^n$.
\end{propositionconvention}

The above statemetn also applies to general $n$-dimensional submanifolds $M,N$ in $\conf_n(\Scal)$, provided, they are modeled as cubes (in an orientation preserving fashion) near their intersections.

\subsection{$\tri^\Nbold$ as subspaces}
Under the natural isomorphism $H^{BM}_n(\conf_n(\Scal))\cong\Hcal_n(\Scal)$ the class $[M]\in H^{BM}_n(\conf_n(\Scal))$ of a properly embedded, oriented, $n$-dimensional submanifold corresponds to the relative class of the compactification $[M^*,*]\in \Hcal_n(\Scal)$. In this light, we will reinterpret the elements $\tri^{\Nbold}(\g_1)\in \Hcal_n(\Scal)$ as proper submanifold classes.

If $\g$ is a smooth, simple, closed loop in $\Xcal$ based at $*$, and $\Nbold=(p_1,\ldots, p_n)$ an ordering on $[n]$, then $\tri^\Nbold(\g)$ defines the proper submanifold of $\conf_n(\Scal)$ that is the restriction of $\g^{(p_1)}\times \cdots \times \g^{(p_n)}$ on the open simplex
$$\mathring{\tri}^n=\{(t_1,\ldots, t_n)\in [0,1]^n:0<t_1<t_2<\cdots< t_n<1\}.$$
As a subspace of $\conf_n(\Scal)$, it contains all configurations where all $n$ particles lie on the curve $\g$ with the strict order $(p_1,\ldots, p_n)$. It receives a canonical orientation by restricting from the $n$-cube. We will use the notation $\tri^\Nbold$ both for th subspace of $\conf_n(\Scal)$ and its homology class in $H^{BM}_n(\conf_n(\Scal))\cong\Hcal_n(\Scal)$. 

Of course, given disjoint $\g$ and $\g'$, the product $\tri^{\Nbold}(\g)\times \tri^{\Nbold'}(\g)$, where $\Nbold,\Nbold'$ is a partition-with-orders of $[n]$ is an oriented, proper, $n$-dimensional submanifold of $\conf_{n}(\Scal)$. Recall that $U_{2g}\subset \Scal$ is the union of the interiors of the pairwise disjoint loops $\alpha_{\pm 1},\ldots, \alpha_{\pm g}$.
\begin{proposition}[\cite{Moriyama, LooijengaStavrou25}]\label{thm:Moriyama}
    For $n\ge 0$, the connected components of $\conf_n(U_{2g})\subset \conf_n(\Scal)$ are the open, oriented, $n$-dimensional, proper submanifolds
    \begin{equation}\label{eq:basisMoriyama}
    \tri^{\Nbold_1}(\alpha_1)\times \tri^{\Nbold_{-1}}(\alpha_{-1})\times \cdots \times \tri^{\Nbold_g}(\alpha_g)\times \tri^{\Nbold_{-g}}(\alpha_{-g})
\end{equation}
    given by partitions-with-orders $\Nbold_1,\Nbold_{-1},\ldots, \Nbold_g,\Nbold_{-g}$ of $[n]$. Furthermore, the elements \eqref{eq:basisMoriyama} form a basis for $\Hcal_n(U_{2g})$.
\end{proposition}
\begin{remark}
    We recall that the injection $\iota: (W_{2g},*)\hookrightarrow (\Xcal,*)$ gives surjections $\Hcal_n(U_{2g})\to \Hcal_n(\Scal)$. So to evaluate all intersection of a closed, oriented, $n$-dimensional submanifold $N\subset \conf_n(\Scal)$ with $\Hcal_n(\Scal)$, it suffices to know how $N$ intersects $\conf_n(U_{2g})$, in which components, and with which signs.
\end{remark}

For non-disjoint $\g,\g'$, the product $\tri^{\Nbold}(\g)\times \tri^{\Nbold'}(\g')$ still makes sense as elements of $\Hcal_{n}(\Scal)$, and a cheap way to obtain a subspace of $\conf_n(\Scal)$ is to ``throw away'' its intersections with $\Xcal^n\setminus \conf_n(\Scal)$. The following formula tells us what we would get if $\g=\g'$.

\begin{proposition}[Lemma 2.1 \cite{LooijengaStavrou25}]\label{prop:exteriorproduct}
    In $\Hcal_{I\sqcup J}(\Scal)$, we have the formula $$\tri^{\Ibold}(\g)\times \tri^{\Jbold}(\g)=\sum_{\Kbold}\sign(\Ibold\Jbold,\Kbold)\tri^{\Kbold}(\g)$$
    where the sum is over all shuffles of $\Ibold$ and $\Jbold$, that is, total orders $\Kbold$ of $I\sqcup J$ which extend the total orders $\Ibold$ and $\Jbold$, and $\sign(\Ibold,\Jbold)$ is the sign of the permutation that takes the concatenated order $\Ibold\Jbold$ to $\Kbold$.
\end{proposition}
\subsection{Toric submanifolds and intersections in $\conf_n(\Scal)$}\label{sec:tori}
The definition of the Subsection \ref{sec:intersectionconventions} also applies to a family  $\gamma_1,\ldots, \gamma_k: S^1\to \Scal$ of pairwise disjoint, simple, closed curves in $\Scal$. Then for any $\sigma\in \Sfrak_k$, we have the oriented, toric submanifold $$\gamma_1^{(\sigma(1))}\times \cdots \times \gamma_n^{(\sigma(k))}:\TT^k=(S^1)^k\to \conf_k(\Scal),$$
see Figure \ref{fig:toric}.

Specialise, now, to the case where the $\gamma_1,\ldots \gamma_k$ are each dual to $\alpha_{i_1},\ldots, \alpha_{i_k}$, respectively, for $i_1,\ldots, i_k\in \{\pm 1,\ldots, \pm g\}$. Assuming the $\g_1,\ldots, \g_k$ are pairwise disjoint, then if $\alpha_i$ appears in this list, then $\alpha_{-i}$ does not. We also assume that all the curves dual to $\alpha_i$ are parallel copies of $\beta_i$ contained in $\Scal_{[i]}$. Then, we obtain $k$ intersection points of the $\gamma_1,\ldots, \gamma_k$ with $U_{2g}$ each labeled by $\sigma(1),\ldots, \sigma(k)$, respectively. Follow each $\alpha_i$ increasingly and record these numbers, to produce a partition-with-orders $\uNbold_{\gamma,\sigma}$ with $2g$, possibly empty, parts.

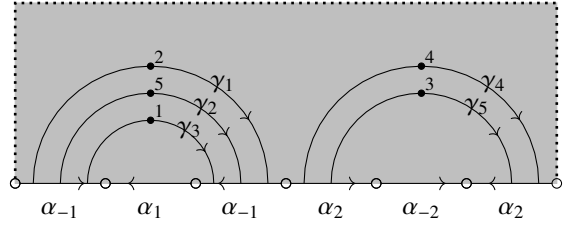
\begin{figure}
  \centering 
   \begin{tikzpicture}[xscale=1.2,yscale=1.2]

    \filldraw[color=lightgray] (0,0) rectangle ++(6,2);
    \draw[dotted, line width=1pt] (0,0.05) to (0,2) to (6,2) to (6,0.05);

    \begin{scope}[shift={(-1,0)}]
    \draw (1,0) circle (1.5pt);
    \begin{scope}[shift={(1,0)}]
        \draw (0,0) circle (1.5pt);
    \draw[postaction={decoration={markings, mark=at position 0.8 with {\arrow{>}}}, decorate}](0.05,0) to ++(0.9,0);
    \draw (0.5,-0.3) node {\footnotesize $\alpha_{-1}$};
    \draw (1,0) circle (1.5pt);
    \end{scope}
    \begin{scope}[shift={(2,0)}]
        \draw (0,0) circle (1.5pt);
    \draw[postaction={decoration={markings, mark=at position 0.3 with {\arrow{<}}}, decorate}](0.05,0) to ++(0.9,0);
    \draw (0.5,-0.3) node {\footnotesize $\alpha_1$};
    \draw (1,0) circle (1.5pt);
    \begin{scope}[shift={(1,0)}]
        \draw (0,0) circle (1.5pt);
    \draw[postaction={decoration={markings, mark=at position 0.3 with {\arrow{<}}}, decorate}](0.05,0) to ++(0.9,0);
    \draw (0.5,-0.3) node {\footnotesize $\alpha_{-1}$};
    \draw (1,0) circle (1.5pt);
    \end{scope}
    \end{scope}


    \draw[domain=0:180, postaction={decoration={markings, mark=at position 0.2 with {\arrow{<}}}, decorate}] plot ({2.5+1.3*cos(\x)}, {1.3*sin(\x)});
    \draw (3.3,1.1) node {\footnotesize $\g_{1}$};

    \draw[domain=0:180, postaction={decoration={markings, mark=at position 0.2 with {\arrow{<}}}, decorate}] plot ({2.5+cos(\x)}, {sin(\x)});
    \draw ({(3.3-2.5)/1.3+2.5},{1.1/1.3}) node {\footnotesize $\g_{2}$};
    
    \draw[domain=0:180, postaction={decoration={markings, mark=at position 0.2 with {\arrow{<}}}, decorate}] plot ({2.5+0.7*cos(\x)}, {0.7*sin(\x)});
    \draw ({(3.3-2.5)/1.3*0.7+2.5},{1.1/1.3*0.7}) node {\footnotesize $\g_{3}$};

    \filldraw (2.5,1.3) circle (1pt);
    \draw (2.6,1.4) node {\footnotesize $_2$};
    \filldraw (2.5,1) circle (1pt);
    \draw (2.6,1.1) node {\footnotesize $_5$};
    \filldraw (2.5,0.7) circle (1pt);
    \draw (2.6,0.8) node {\footnotesize $_1$};

    \begin{scope}[shift={(4,0)}]
    \begin{scope}[shift={(-1,0)}]
            \draw[domain=0:180, postaction={decoration={markings, mark=at position 0.2 with {\arrow{<}}}, decorate}] plot ({2.5+1.3*cos(\x)}, {1.3*sin(\x)});
    \draw (3.3,1.1) node {\footnotesize $\g_{4}$};

    \draw[domain=0:180, postaction={decoration={markings, mark=at position 0.2 with {\arrow{<}}}, decorate}] plot ({2.5+cos(\x)}, {sin(\x)});
    \draw ({(3.3-2.5)/1.3+2.5},{1.1/1.3}) node {\footnotesize $\g_{5}$};
    
    \filldraw (2.5,1.3) circle (1pt);
    \draw (2.6,1.4) node {\footnotesize $_4$};
    \filldraw (2.5,1) circle (1pt);
    \draw (2.6,1.1) node {\footnotesize $_3$};
    \end{scope}
    
        \draw (0,0) circle (1.5pt);
    \draw[postaction={decoration={markings, mark=at position 0.8 with {\arrow{>}}}, decorate}](0.05,0) to ++(0.9,0);
    \draw (0.5,-0.3) node {\footnotesize $\alpha_{2}$};
    \draw (1,0) circle (1.5pt);
    
    \begin{scope}[shift={(1,0)}]
        \draw (0,0) circle (1.5pt);
    \draw[postaction={decoration={markings, mark=at position 0.8 with {\arrow{>}}}, decorate}](0.05,0) to ++(0.9,0);
    \draw (0.5,-0.3) node {\footnotesize $\alpha_{-2}$};
    \draw (1,0) circle (1.5pt);
    \end{scope}
    
    \begin{scope}[shift={(2,0)}]
    
        \draw (0,0) circle (1.5pt);
    \draw[postaction={decoration={markings, mark=at position 0.3 with {\arrow{<}}}, decorate}](0.05,0) to ++(0.9,0);
    \draw (0.5,-0.3) node {\footnotesize $\alpha_{2}$};
    \draw (1,0) circle (1.5pt);
    
    \end{scope}
    \end{scope}
    \end{scope}

   \end{tikzpicture}
   \caption{\footnotesize A $5$-torus in  $\conf_5(U)$. The submanifold $\g_1^{(2)}\times \g_2^{(5)}\times \g_3^{(1)}\times \g_4^{(4)}\times \g_5^{(3)}$ intersects $\conf_5(U_1)$ precisely once transversally in the component $\tri^{(2,5,1)}(\alpha_{-1})\times \tri^{(4,3)}(\alpha_2)$.}
   \label{fig:toric}
\end{figure}

\begin{proposition}\label{prop:toricintersections}
    The toric submanifold $\gamma_1^{(\sigma(1))}\times\cdots \times \gamma_k^{(\sigma(k))}: \TT^k\to \conf_k(\Scal)$ is supported in the union $\Scal_{[i_1]}\cup\cdots \Scal_{[i_k]}$ and intersects $\conf_k(U_{2g})$ precisely once, transversally, in the component $\tri^{\uNbold_{\gamma,\sigma}}(\alpha_1\times \alpha_{-1}\times \cdots \times \alpha_{g}\times \alpha_{-g})$, which lies in the open $k$-cube $\tri^{(\sigma(1))}(\alpha_{i_1})\times \cdots \times \tri^{(\sigma(k))}(\alpha_{i_k})$.
\end{proposition}

\subsection{Embedded surfaces with boundary}\label{subsec:Bianchitrick}
Here is an ingredient for constructing non-toric closed submanifolds due to Bianchi (see Example 3 from Section 4.1 of \cite{Bianchi2020}, and having also a counterpart in \cite{bianchistavrou22Johnson}).

\begin{figure}
  \centering 
     \begin{tikzpicture}[xscale=1.2,yscale=1.2]

    \filldraw[color=lightgray] (0,0) rectangle ++(4,1.5);
    \draw[dotted, line width=1pt] (0,0.05) to (0,1.5) to (4,1.5) to (4,0.05);

    \draw (0,0) circle (1.5pt);
    \draw[postaction={decoration={markings, mark=at position 0.8 with {\arrow{>}}}, decorate}](0.05,0) to ++(0.9,0);
    \draw (0.5,-0.3) node {\footnotesize $\alpha_i$};
    \draw (1,0) circle (1.5pt);
    \begin{scope}[shift={(1,0)}]
        \draw (0,0) circle (1.5pt);
    \draw[postaction={decoration={markings, mark=at position 0.8 with {\arrow{>}}}, decorate}](0.05,0) to ++(0.9,0);
    \draw (0.5,-0.3) node {\footnotesize $\alpha_{-i}$};
    \draw (1,0) circle (1.5pt);
    \end{scope}
    \begin{scope}[shift={(2,0)}]
        \draw (0,0) circle (1.5pt);
    \draw[postaction={decoration={markings, mark=at position 0.3 with {\arrow{<}}}, decorate}](0.05,0) to ++(0.9,0);
    \draw (0.5,-0.3) node {\footnotesize $\alpha_i$};
    \draw (1,0) circle (1.5pt);
    \begin{scope}[shift={(1,0)}]
        \draw (0,0) circle (1.5pt);
    \draw[postaction={decoration={markings, mark=at position 0.3 with {\arrow{<}}}, decorate}](0.05,0) to ++(0.9,0);
    \draw (0.5,-0.3) node {\footnotesize $\alpha_{-i}$};
    \draw (1,0) circle (1.5pt);
    \end{scope}
    \end{scope}

    \draw[domain=0:180, postaction={decoration={markings, mark=at position 0.75 with {\arrow{<}}}, decorate}] plot ({1.5+cos(\x)}, {sin(\x)});
    \draw (0.63,0.8) node {\footnotesize $\beta_i$};

    \draw[domain=0:180, postaction={decoration={markings, mark=at position 0.4 with {\arrow{<}}}, decorate}] plot ({2.5+cos(\x)}, {sin(\x)});
    \draw (3.2,1.1) node {\footnotesize $\beta_{-i}$};

    \filldraw (1.5,1) circle (1pt);
    \draw (1.6,1.1) node {\footnotesize $_1$};
    \filldraw (2.5,1) circle (1pt);
    \draw (2.6,1.1) node {\footnotesize $_2$};

    \filldraw[color=lightgray] (2,{sqrt(3)/2}) circle (2pt);
    \draw[color=darkgray]  (2,{sqrt(3)/2}) circle (2pt);
    
   \end{tikzpicture}
\caption{The embedded torus-minus-a-disc $\Sigma^{(1,2)}(\beta_i\times \beta_{-i}$).}
\label{fig:Sigma_beta_ixbeta_-i}
\end{figure}
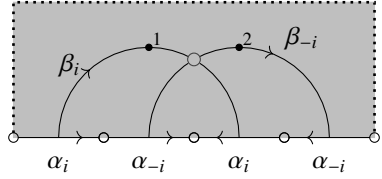

Suppose the simple closed curves $\gamma_1,\gamma_2:S^1\to \Scal$ intersect transversely once, positively, at the point $p$. Then the torus $\gamma_1\times\gamma_2$ fails to lie in $\conf_2(\Scal)$ because it intersects the diagonal of $\Scal\times \Scal$ at $(p,p)$. We obtain a (non-closed) submanifold of $\conf_n(\Scal)$ by ``throwing away'' the intersection with the diagonal. More precisely, we remove from $\TT^2$ the interior of a small disc centred in the preimage $q$ of $(p,p)$ to obtain an embedding of $\TT^2\setminus \mathring{D}^2$ in $\conf_2(\Scal)$. After a small isotopy, we can make the boundary $S^1\cong \partial(\TT^2\setminus \mathring{D}^2)=\overline{\partial(D^2)}$ to parametrise the loop $v^{(1,2)}(p)$ where particle $1$ sits at point $p$ and particle $2$ performs a small counterclockwise $2\pi$-orbit around $1$. Perform this isotopy near the boundary of $\TT^2\setminus \mathring{D}^2$, to obtain an embedding $\Sigma(\gamma_1\times \gamma_2): \TT^2\setminus \mathring{D}^2\to \conf_2(\Scal)$. 

Suppose we started with the curves $\beta_i,\beta_{-i}$, then we can assume that our isotopy takes place very locally and that $\Sigma(\beta_i\times \beta_{-i})$ has the properties: (i) its image is supported in $\Scal_{[i]}$, (ii) away from a neighbourhood of the boundary it agrees with the product embedding $\beta_i\times \beta_{-i}$, and (iii) the loop $v^{(1,2)}(p_i)$ is along the circle of a very small radius $\varepsilon>0$ centred at $p_i$. 

\begin{proposition}\label{prop:boundarysubsurfaceintersections}
    The submanifold $\Sigma(\beta_i\times \beta_{-i})$ of $\conf_2(\Scal)$ is supported in $\conf_2(\Scal_{[i]})$, and it intersects $\conf_2(U_{2g})$ transversally precisely once in the component $\tri^{(1)}(\alpha_i)\times \tri^{(2)}(\alpha_{-i})$, with sign of the intersection $\langle \tri^{(1)}(\alpha_i)\times \tri^{(2)}(\alpha_{-i}),\Sigma(\beta_i\times \beta_{-i})\rangle =- 1$.
\end{proposition}
\begin{proof}
    We only need to justify the sign  of the intersection.
    Since the piece we removed from the torus and the small isotopy in the definition of $\Sigma(\beta_i\times \beta_{-i})$ are supported away from $U_{2g}$, the required intersection is modelled after
    $$\langle \alpha_i^{(1)}\times \alpha_{-i}^{(2)}, \beta_i^{(1)}\times \beta_{-i}^{(2)}\rangle.$$
    By Proposition \ref{prop:intersectionconvention}, this is computed as $$-\langle \alpha_i, \beta_i \rangle\cdot \langle \alpha_{-i}, \beta_{-i}\rangle =-(+1)(+1)=-1.$$
\end{proof}


\subsection{Auxiliary tubes}\label{sec:tubes} 
Given two points $p,q\in \Scal$ at least $\varepsilon$ away from $U_{2g}$, the circles 
$v^{(1,2)}(p)$ and $v^{(1,2)}(q)$ from Section \ref{subsec:Bianchitrick} are isotopic, and any path $\epsilon: p_i\rightsquigarrow p_{i+1}$, gives an explicit isotopy cylinder
\begin{align}
    v^{(1,2)}(\epsilon):[0,1]\times S^1&\to \conf_2(\Scal),\label{eq:tubedefn} \\
    (t,\theta)&\mapsto (\epsilon(t), v_{1,2}(\epsilon(t))) \notag
\end{align}
with boundary $\bar{v}^{(1,2)}(p)\sqcup v^{(1,2)}(q)$.  For example, $v^{(1,2)}(\varepsilon_{i,i+1})$ glued appropriately with $\Sigma(\beta_i\times \beta_{-i})$ and $\Sigma(\beta_{i+1}\times \beta_{-(i+1)})$ forms an embedded surface of genus $2$ in $\conf_2(\Scal)$.

\begin{figure}[h]
  \centering 
   \begin{tikzpicture}[xscale=1.2,yscale=1.2]

    \filldraw[color=lightgray] (0,0) rectangle ++(3,2);
    \filldraw[color=gray] (0.3,0) to (0.3,1.5) to (0.7,1.5) to (.7,0);
    \filldraw[color=gray] (0.5,1.5) circle (0.2);
    \filldraw[color=gray] (2.3,0) to (2.3,1.5) to (2.7,1.5) to (2.7,0);
    \filldraw[color=gray] (2.5,1.5) circle (0.2);

    \draw[dotted, line width=1pt] (0,0.05) to (0,2) to (3,2) to (3,0.05);

    \begin{scope}[shift={(0,0.15)}]
        
        
        \begin{scope}[shift={(0,-0.6)}, opacity=0.4]
        \filldraw (0.5,0.5) circle (1pt);
        \draw[domain=0:360, postaction={decoration={markings, mark=at position 1 with {\arrow{>}}}, decorate}] plot ({0.5+0.2*cos(\x)}, {0.5+0.2*sin(\x)});
        \filldraw (0.3,0.5) circle (1pt);      
        \end{scope}

        \begin{scope}[shift={(0,-0.75)}, opacity=0.25]
        \draw[domain=0:360, postaction={decoration={markings, mark=at position 1 with {\arrow{>}}}, decorate}] plot ({0.5+0.2*cos(\x)}, {0.5+0.2*sin(\x)});     
        \end{scope}

        \begin{scope}[opacity=0.4]
            
            \draw[domain=0:360, postaction={decoration={markings, mark=at position 0.1 with {\arrow{>}}}, decorate}] plot ({2.5+0.2*cos(180-\x)}, {-0.20+0.2*sin(180-\x)});
        \end{scope}
        
        \begin{scope}[opacity=0.3]
            \filldraw (2.5,-0.05) circle (1pt);
            \draw[domain=0:360, postaction={decoration={markings, mark=at position 0.1 with {\arrow{>}}}, decorate}] plot ({2.5+0.2*cos(180-\x)}, {-0.05+0.2*sin(180-\x)});
            \filldraw (2.7,-0.05) circle (1pt);
        \end{scope}
        
        \begin{scope}[opacity=0.1]
        \filldraw (2.5,0.1) circle (1pt);
        \draw[domain=0:360, postaction={decoration={markings, mark=at position 0.1 with {\arrow{>}}}, decorate}] plot ({2.5+0.2*cos(180-\x)}, {0.1+0.2*sin(180-\x)});
        \filldraw (2.7,0.1) circle (1pt);
        \end{scope}

    \end{scope}

    \filldraw[color=white] (0.1,-0.012) rectangle ++(0.8,-0.3);
    \filldraw[color=white] (2.1,-0.012) rectangle ++(0.8,-0.3);

    \begin{scope}[shift={(-1,0)}]
    \draw (1,0) circle (1.5pt);
    \begin{scope}[shift={(1,0)}]
        \draw (0,0) circle (1.5pt);
    \draw[postaction={decoration={markings, mark=at position 0.8 with {\arrow{>}}}, decorate}](0.05,0) to ++(0.9,0);
    \draw (0.5,-0.3) node {\footnotesize $\alpha_{i}$};
    \draw (1,0) circle (1.5pt);
    \end{scope}

    \begin{scope}[shift={(2,0)}]
        \draw (0.05,0) to (0.3,0);
        \draw[dotted, line width=1pt] (0.3,0) to (0.7,0);
        \draw (0.7,0) to (0.95,0);
        
        \draw (1,0) circle (1.5pt);
        
    \begin{scope}[shift={(1,0)}]
        \draw (0,0) circle (1.5pt);
    \draw[postaction={decoration={markings, mark=at position 0.3 with {\arrow{<}}}, decorate}](0.05,0) to ++(0.9,0);
    \draw (0.5,-0.3) node {\footnotesize $\alpha_{i}$};
    \draw (1,0) circle (1.5pt);
    \end{scope}
    
    \end{scope}

    \begin{scope}[shift={(1,0)}]
        \draw[postaction={decoration={markings, mark=at position 0.3 with {\arrow{>}}}, decorate}] (0.5,1.5) to (0.5,0);
        \draw[postaction={decoration={markings, mark=at position 0.5 with {\arrow{<}}}, decorate}] (2.5,1.5) to (2.5,0);
        \draw (2.3,0.75) node {\footnotesize $\epsilon$};

    \begin{scope}[shift={(0,-0.15)}]
        \filldraw (0.5,0.5) circle (1pt);
        \draw (0.42,0.56) node {\footnotesize $_1$};
        \draw[domain=0:360, postaction={decoration={markings, mark=at position 1 with {\arrow{>}}}, decorate}] plot ({0.5+0.2*cos(\x)}, {0.5+0.2*sin(\x)});
        \filldraw (0.3,0.5) circle (1pt);
        \draw (0.2,0.56) node {\footnotesize $_2$};

        \begin{scope}[shift={(0,0.15)}, opacity=0.6]
            \filldraw (0.5,0.5) circle (1pt);
            \draw[domain=0:360, postaction={decoration={markings, mark=at position 1 with {\arrow{>}}}, decorate}] plot ({0.5+0.2*cos(\x)}, {0.5+0.2*sin(\x)});
            \filldraw (0.3,0.5) circle (1pt);      
        \end{scope}
        \begin{scope}[shift={(0,0.3)}, opacity=0.4]
            \filldraw (0.5,0.5) circle (1pt);
            \draw[domain=0:360, postaction={decoration={markings, mark=at position 1 with {\arrow{>}}}, decorate}] plot ({0.5+0.2*cos(\x)}, {0.5+0.2*sin(\x)});
            \filldraw (0.3,0.5) circle (1pt);      
        \end{scope}
        \begin{scope}[shift={(0,0.45)}, opacity=0.2]
            \filldraw (0.5,0.5) circle (1pt);
            \draw[domain=0:360, postaction={decoration={markings, mark=at position 1 with {\arrow{>}}}, decorate}] plot ({0.5+0.2*cos(\x)}, {0.5+0.2*sin(\x)});
            \filldraw (0.3,0.5) circle (1pt);      
        \end{scope}
        \begin{scope}[shift={(0,0.6)}, opacity=0.1]
            \filldraw (0.5,0.5) circle (1pt);
            \draw[domain=0:360, postaction={decoration={markings, mark=at position 1 with {\arrow{>}}}, decorate}] plot ({0.5+0.2*cos(\x)}, {0.5+0.2*sin(\x)});
            \filldraw (0.3,0.5) circle (1pt);      
        \end{scope}
    
        \begin{scope}[shift={(0,-0.15)}, opacity=0.6]
            \filldraw (0.5,0.5) circle (1pt);
            \draw[domain=0:360, postaction={decoration={markings, mark=at position 1 with {\arrow{>}}}, decorate}] plot ({0.5+0.2*cos(\x)}, {0.5+0.2*sin(\x)});
            \filldraw (0.3,0.5) circle (1pt);      
        \end{scope}
        
    \end{scope}
    \end{scope}
    \end{scope}

   \end{tikzpicture}
   \caption{\footnotesize An animation of the tube $v_{1,2}(\epsilon)$: particle $1$ is traversing arc $\epsilon$ while particle $2$ is orbiting $1$ at a fixed radius $\varepsilon$.}
   \label{fig:tubeanimation}
\end{figure}
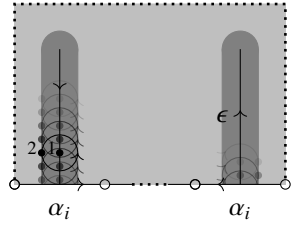

\begin{lemma}\label{lem:tubeintersection}
    The submanifold $v^{(1,2)}(\epsilon)$ of $\conf_2(\Scal)$ is supported in an $\varepsilon$-neighbourhood of the arc $\epsilon$. If $\epsilon$ is disjoint from $U_{2g}$ and $\varepsilon$ is small enough, then $v^{(1,2)}(\epsilon)$ is disjoint from $\conf_2(U_{2g})$. If, on the other hand, 
    $\epsilon$ intersects $U_{2g}$ transversely once along $\alpha_i$ so that $\langle \epsilon, \alpha_i\rangle$ is positive, then $v^{(1,2)}(\epsilon)$
    intersects $\conf_2(U_{2g})$ transversely twice, once in each of the components $\tri^{(1,2)}(\alpha_i)$ and $\tri^{(2,1)}(\alpha_i)$ with the equal signs
    \begin{equation*}
        \langle \tri^{(1,2)}(\alpha_i) , v^{(1,2)}(\epsilon)\rangle =
        \langle \tri^{(2,1)}(\alpha_i) , v^{(1,2)}(\epsilon)\rangle =+1.
    \end{equation*}
\end{lemma}
\begin{proof}
    The first assertion is obvious. If $\epsilon$ is disjoint from $U_{2g}$, then by compactness the two subspaces are a positive distance away; take $\varepsilon$ to be smaller than this.
    In the second case, assume the intersection between $\alpha_i$ and $\epsilon$ is $p$. Then the intersection between $\conf_2(U_{2g})$ and $v^{(1,2)}(\epsilon)$ occur at the two points $(p,p\pm\varepsilon)\in \conf_2(U_{2g})\subset \conf_2(\RR)$. To find the signs, we follow the recipe of Section \ref{sec:intersectionconventions}.

     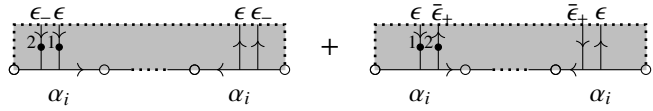
\begin{figure}
  \centering 
   \begin{tikzpicture}[xscale=1.2,yscale=1.2]

    \begin{scope}[shift={(-3,0)}]
    \filldraw[color=lightgray] (0,0) rectangle ++(3,0.5);
    \draw[dotted, line width=1pt] (0,0.05) to (0,0.5) to (3,0.5) to (3,0.05);
        
    \begin{scope}[shift={(-1,0)}]
    \draw (1,0) circle (1.5pt);
    \begin{scope}[shift={(1,0)}]
        \draw (0,0) circle (1.5pt);
    \draw[postaction={decoration={markings, mark=at position 0.8 with {\arrow{>}}}, decorate}](0.05,0) to ++(0.9,0);
    \draw (0.5,-0.3) node {\footnotesize $\alpha_{i}$};
    \draw (1,0) circle (1.5pt);
    \end{scope}
    \begin{scope}[shift={(2,0)}]
        \draw (0.05,0) to (0.3,0);
        \draw[dotted, line width=1pt] (0.3,0) to (0.7,0);
        \draw (0.7,0) to (0.95,0);
        
    \draw (1,0) circle (1.5pt);
    \begin{scope}[shift={(1,0)}]
        \draw (0,0) circle (1.5pt);
    \draw[postaction={decoration={markings, mark=at position 0.3 with {\arrow{<}}}, decorate}](0.05,0) to ++(0.9,0);
    \draw (0.5,-0.3) node {\footnotesize $\alpha_{i}$};
    \draw (1,0) circle (1.5pt);
    \end{scope}
    
    \end{scope}

    \begin{scope}[shift={(1,0)}]
    
        \draw[postaction={decoration={markings, mark=at position 0.3 with {\arrow{>}}}, decorate}] (0.5,0.5) to (0.5,0);
    \draw[postaction={decoration={markings, mark=at position 0.5 with {\arrow{<}}}, decorate}] (2.5,0.5) to (2.5,0);
     \draw[postaction={decoration={markings, mark=at position 0.3 with {\arrow{>}}}, decorate}] (0.3,0.5) to (0.3,0);
    \draw[postaction={decoration={markings, mark=at position 0.5 with {\arrow{<}}}, decorate}] (2.7,0.5) to (2.7,0);
    
    \draw (2.5,0.62) node {\footnotesize $\epsilon$};
    \draw (2.75,0.6) node {\footnotesize $\epsilon_-$};
    \draw (0.5,0.62) node {\footnotesize $\epsilon$};
    \draw (0.3,0.6) node {\footnotesize $\epsilon_-$};
    
    \filldraw (0.5,0.25) circle (1pt);
    \draw (0.42,0.31) node {\footnotesize $_1$};
   
    \filldraw (0.3,0.25) circle (1pt);
    \draw (0.2,0.31) node {\footnotesize $_2$};
    \end{scope}
    \end{scope}
    \end{scope}

    \draw (0.5,0.25) node {$+$};

    \begin{scope}[shift={(1,0)}]
            \filldraw[color=lightgray] (0,0) rectangle ++(3,0.5);
    \draw[dotted, line width=1pt] (0,0.05) to (0,0.5) to (3,0.5) to (3,0.05);
        
    \begin{scope}[shift={(-1,0)}]
    
    \draw (1,0) circle (1.5pt);
    \begin{scope}[shift={(1,0)}]
        \draw (0,0) circle (1.5pt);
    \draw[postaction={decoration={markings, mark=at position 0.9 with {\arrow{>}}}, decorate}](0.05,0) to ++(0.9,0);
    \draw (0.5,-0.3) node {\footnotesize $\alpha_{i}$};
    \draw (1,0) circle (1.5pt);
    \end{scope}
    \begin{scope}[shift={(2,0)}]
        \draw (0.05,0) to (0.3,0);
        \draw[dotted, line width=1pt] (0.3,0) to (0.7,0);
        \draw (0.7,0) to (0.95,0);
        
    \draw (1,0) circle (1.5pt);
    \begin{scope}[shift={(1,0)}]
        \draw (0,0) circle (1.5pt);
    \draw[postaction={decoration={markings, mark=at position 0.2 with {\arrow{<}}}, decorate}](0.05,0) to ++(0.9,0);
    \draw (0.5,-0.3) node {\footnotesize $\alpha_{i}$};
    \draw (1,0) circle (1.5pt);
    \end{scope}
    
    \end{scope}

    \begin{scope}[shift={(1,0)}]
        \draw[postaction={decoration={markings, mark=at position 0.3 with {\arrow{>}}}, decorate}] (0.5,0.5) to (0.5,0);
    \draw[postaction={decoration={markings, mark=at position 0.5 with {\arrow{<}}}, decorate}] (2.5,0.5) to (2.5,0);
    
     \draw[postaction={decoration={markings, mark=at position 0.3 with {\arrow{<}}}, decorate}] (0.7,0.5) to (0.7,0);
    \draw[postaction={decoration={markings, mark=at position 0.5 with {\arrow{>}}}, decorate}] (2.3,0.5) to (2.3,0);
    
    \draw (2.5,0.62) node {\footnotesize $\epsilon$};
    \draw (2.25,0.62) node {\footnotesize $\bar{\epsilon}_+$};
    \draw (0.45,0.62) node {\footnotesize $\epsilon$};
    \draw (0.75,0.62) node {\footnotesize $\bar{\epsilon}_+$};

    \filldraw (0.5,0.25) circle (1pt);
    \draw (0.42,0.31) node {\footnotesize $_1$};
   
    \filldraw (0.7,0.25) circle (1pt);
    \draw (0.6,0.31) node {\footnotesize $_2$};
    \end{scope}
    \end{scope}

    \end{scope}
   \end{tikzpicture}
   \caption{\footnotesize The intersection of $v_{1,2}(\epsilon)$ with $\conf_2(U_1)$ is the same as that of $\epsilon^{(1)}\times \epsilon_-^{(2)}+\epsilon^{(1)}\times \bar{\epsilon}_+^{(2)}$ depicted here.}
   \label{fig:tubeintersections}
\end{figure}
     Let $\epsilon_{\pm}=\epsilon\pm \varepsilon$ be horizontal shifts of the arc $\epsilon$ so that the three arcs intersect $\alpha_{i}$ in the order $\epsilon_-, \epsilon, \epsilon_+$. Then, up to a small isotopy $v_{1,2}(\epsilon)$ is modeled near its intersections with $U_{2g}$ by the rectangles $\epsilon^{(1)}\times \bar\epsilon_+^{(2)}$ and $ \epsilon^{(1)}\times \epsilon_-^{(2)}$; this is by recalling the parametrisation of $v^{(1,2)}(\epsilon)$ with a counterclockwise orbit \eqref{eq:tubedefn} (cf. Figures \ref{fig:tubeanimation} and \ref{fig:tubeintersections}). The former intersects $\conf_2(U_{2g})$ in $\tri^{(1,2)}(\alpha_i)$ and the latter in $\tri^{(2,1)}(\alpha_i)$; both intersections are transversal at a single point. The signs are given by  Proposition \ref{prop:intersectionconvention} and the assumed fact $\langle \epsilon, \alpha_i\rangle=1$ as
     \begin{align*}
         \langle \tri^{(1,2)}(\alpha_i), \epsilon^{(1)}\times \bar\epsilon_+^{(2)}\rangle&=\langle \alpha_i^{(1)}\times \alpha_i^{(2)}, \epsilon^{(1)}\times \bar\epsilon_+^{(2)}\rangle\\
         &=-\langle \alpha_i^{(1)}, \epsilon^{(1)}\rangle\cdot \langle \alpha_i^{(2)}, \bar\epsilon_+^{(2)}\rangle\\
         &=-(-1)(+1)\\
         &=+1,
     \end{align*}
     and
     \begin{align*}
         \langle \tri^{(2,1)}(\alpha_i), \epsilon^{(1)}\times \epsilon_-^{(2)}\rangle&=\langle \alpha_i^{(2)}\times \alpha_i^{(1)}, \epsilon^{(1)}\times \epsilon_-^{(2)}\rangle\\
         &=+\langle \alpha_i^{(1)}, \epsilon^{(1)}\rangle\cdot \langle \alpha_i^{(2)}, \epsilon_-^{(2)}\rangle\\
         &=+(-1)(-1)\\
         &=+1.
     \end{align*}
\end{proof}

\begin{remark}
    We could justify a priori why the two signs in the last lemma must be equal. The tube $v_{1,2}(\epsilon)$ is null-bordant in $\Scal^2$ as it bounds the solid cylinder when letting the radius $\varepsilon$ shrink to zero. So its algebraic intersection with any proper submanifold of $\conf_2(\Scal)$ that remains proper in $\Scal^2$ must vanish. Proposition \ref{prop:exteriorproduct} gives the decomposition of the open square $\alpha_i^{(1)}\times \alpha_i^{(2)}=\tri^{(1,2)}(\alpha_i)-\tri^{(2,1)}(\alpha_i)$. As the intersection of $v_{1,2}(\epsilon)$ with $\alpha_i^{(1)}\times \alpha_i^{(2)}$ vanishes, its intersections with $\tri^{(1,2)}(\alpha_i)$  and $\tri^{(2,1)}(\alpha_i)$ are equal.
\end{remark}

\subsection{The three particle entanglement $\Ecal$}\label{sec:Ecal} We will now combine the constructions from Sections \ref{sec:tori}, \ref{subsec:Bianchitrick} and \ref{sec:tubes} to construct a closed, oriented, $3$-dimensional submanifold $\Ecal$ of $\conf_3(\Scal)$. The reader should keep in mind that the aim of $\Ecal$ is to intersect the component $\tri^{(1,2)}(\alpha_1)\times \tri^{(3)}(\alpha_{-1})$ of $\conf_3(U_{2g})$ once. The construction of is teleological, with the aim to make Theorem \ref{thm:sEcalPerfectPairing} true. 

\begin{figure}
  \centering 
   \begin{tikzpicture}[xscale=1.2,yscale=1.2]

    \filldraw[color=lightgray] (0,0) rectangle ++(8,3);
    \draw[dotted, line width=1pt] (0,0.05) to (0,3) to (8,3) to (8,0.05);

    \draw (0,0) circle (1.5pt);
    \draw[postaction={decoration={markings, mark=at position 0.67 with {\arrow{>}}}, decorate}](0.05,0) to ++(0.9,0);
    \draw (0.5,-0.3) node {\footnotesize $\alpha_1$};
    \draw (1,0) circle (1.5pt);
    \begin{scope}[shift={(1,0)}]
        \draw (0,0) circle (1.5pt);
    \draw[postaction={decoration={markings, mark=at position 0.8 with {\arrow{>}}}, decorate}](0.05,0) to ++(0.9,0);
    \draw (0.5,-0.3) node {\footnotesize $\alpha_{-1}$};
    \draw (1,0) circle (1.5pt);
    \end{scope}
    \begin{scope}[shift={(2,0)}]
        \draw (0,0) circle (1.5pt);
    \draw[postaction={decoration={markings, mark=at position 0.43 with {\arrow{<}}}, decorate}](0.05,0) to ++(0.9,0);
    \draw (0.5,-0.3) node {\footnotesize $\alpha_1$};
    \draw (1,0) circle (1.5pt);
    \begin{scope}[shift={(1,0)}]
        \draw (0,0) circle (1.5pt);
    \draw[postaction={decoration={markings, mark=at position 0.3 with {\arrow{<}}}, decorate}](0.05,0) to ++(0.9,0);
    \draw (0.5,-0.3) node {\footnotesize $\alpha_{-1}$};
    \draw (1,0) circle (1.5pt);
    \end{scope}
    \end{scope}

    \draw[domain=0:180, postaction={decoration={markings, mark=at position 0.75 with {\arrow{<}}}, decorate}] plot ({1.5+cos(\x)}, {sin(\x)});
    \draw (0.63,0.8) node {\footnotesize $\beta_1$};

    \draw[color=red, domain=0:180, postaction={decoration={markings, mark=at position 0.75 with {\arrow{<}}}, decorate}] plot ({1.5+0.8*cos(\x)}, {0.8*sin(\x)});
    \draw (1,0.35) node {\color{red} \footnotesize $\beta_1'$};

    \draw (1.92,{sqrt(3)/2+0.22}) node {\footnotesize $p_1$};
    \draw (1.82,0.5) node {\color{teal} \footnotesize $p_1'$};

    \draw[color=blue, looseness=1, postaction={decoration={markings, mark=at position 0.8 with {\arrow{>}}}, decorate}] ({2.5-cos(2*asin(0.4))},{0.8*sin(acos(0.4))}) to[out=180, in=90] (1.25,0);
    
    \draw[color=blue, looseness=1.1, postaction={decoration={markings, mark=at position 0.3 with {\arrow{>}}}, decorate}] (3.75,0) to[out=90, in=120] (6,{sqrt(3)/2});
    
    \draw (4,1) node {\color{blue}\footnotesize $\epsilon'$};

    \draw[domain=0:180, postaction={decoration={markings, mark=at position 0.4 with {\arrow{<}}}, decorate}] plot ({2.5+cos(\x)}, {sin(\x)});
    \draw (3.2,1.1) node {\footnotesize $\beta_{-1}$};
    \filldraw (2,{sqrt(3)/2}) circle (1pt);

    \filldraw[color=teal] ({2.5-cos(2*asin(0.4))},{0.8*sin(acos(0.4))}) circle (1pt);

    \begin{scope}[shift={(4,0)}]
    
        \draw (0,0) circle (1.5pt);
    \draw[postaction={decoration={markings, mark=at position 0.8 with {\arrow{>}}}, decorate}](0.05,0) to ++(0.9,0);
    \draw (0.5,-0.3) node {\footnotesize $\alpha_{2}$};
    \draw (1,0) circle (1.5pt);
    
    \begin{scope}[shift={(1,0)}]
        \draw (0,0) circle (1.5pt);
    \draw[postaction={decoration={markings, mark=at position 0.8 with {\arrow{>}}}, decorate}](0.05,0) to ++(0.9,0);
    \draw (0.5,-0.3) node {\footnotesize $\alpha_{-2}$};
    \draw (1,0) circle (1.5pt);
    \end{scope}
    
    \begin{scope}[shift={(2,0)}]
    
        \draw (0,0) circle (1.5pt);
    \draw[postaction={decoration={markings, mark=at position 0.3 with {\arrow{<}}}, decorate}](0.05,0) to ++(0.9,0);
    \draw (0.5,-0.3) node {\footnotesize $\alpha_{2}$};
    \draw (1,0) circle (1.5pt);
    
    \begin{scope}[shift={(1,0)}]
    
        \draw (0,0) circle (1.5pt);
    \draw[postaction={decoration={markings, mark=at position 0.3 with {\arrow{<}}}, decorate}](0.05,0) to ++(0.9,0);
    \draw (0.5,-0.3) node {\footnotesize $\alpha_{-2}$};
    \draw (1,0) circle (1.5pt);
    
    \end{scope}
    \end{scope}
    \draw[domain=0:180, postaction={decoration={markings, mark=at position 0.6 with {\arrow{<}}}, decorate}] plot ({1.5+cos(\x)}, {sin(\x)});
    \draw (1.1,0.7) node {\footnotesize $\beta_{2}$};
    \draw[domain=0:180, postaction={decoration={markings, mark=at position 0.4 with {\arrow{<}}}, decorate}] plot ({2.5+cos(\x)}, {sin(\x)});
    \draw (3.3,1.1) node {\footnotesize $\beta_{-2}$};
    \filldraw (2,{sqrt(3)/2}) circle (1pt);
    \draw (2,{sqrt(3)/2-0.2}) node {\footnotesize $p_{2}$};
    \end{scope}
    
    \draw[domain={atan(sqrt(3)/4)}:{180-atan(sqrt(3)/4)}, postaction={decoration={markings, mark=at position 0.4 with {\arrow{<}}}, decorate}] plot ({4+sqrt(19/4)*cos(\x)}, {sqrt(19/4)*sin(\x)});
    \draw (4.3,2.35) node {\footnotesize $\epsilon$};

   \end{tikzpicture}
   \caption{\footnotesize The curves involved in $\Ecal$. The curve $\beta_1'$ is parallel to $\beta_1$ and intersects $\alpha_1$ after $\beta_1$ does. The intersection $p_1'$ is between $\beta_1'$ and $\beta_{-1}$. The arc $\epsilon'$ is from $p_1'$ to $p_2$; it intersects none of the other curves except at its boundary, and intersects only the arc $\alpha_{-1}$, with $(\epsilon',\alpha_1)$ positive.}
   \label{fig:inputforEcal}
\end{figure}
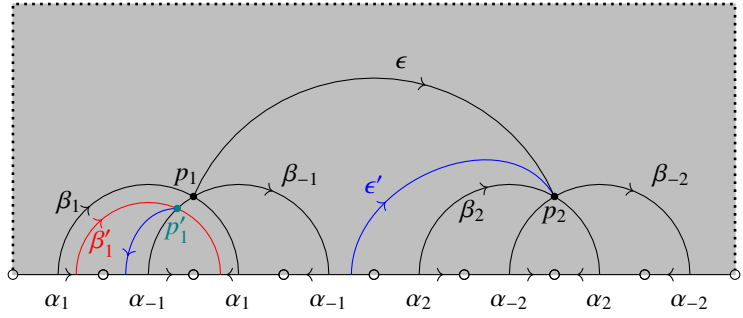

Referring the reader to Figure \ref{fig:inputforEcal}, take a parallel copy $\beta_1'$ of $\beta_1$ that intersects $\beta_{-1}$ at $p_1'$, and so that $\alpha_1$ intersects first $\beta_1$ and then $\beta_1'$. Consider the $3$-torus $\beta_1^{(1)}\times\beta_{-1}^{(3)}\times\beta_1'^{(2)}$ (that intersects $\tri^{(1,2)}(\alpha_1)\times \tri^{(3)}(\alpha_{-1})$ transversally once). This torus fails to be contained in $\conf_3(\Scal)$ because it intersects the $13$-diagonal along the curve $(p_1,p_1)\times \beta_1'^{(2)}$
and the $23$-diagonal along $\beta_1^{(1)}\times (p_1',p_1')$. As in Section \ref{subsec:Bianchitrick}, we take a small disc $D^2$ around the $(\beta_1^{(1)}\times\beta_{-1}^{(3)})$-preimage of $(p_1,p_1)$ and remove the open solid torus $\mathring{D}^2\times S^1$ from $S^1\times S^1\times S^1$ to obtain a manifold with a boundary torus $S^1\times S^1$; after a small isotopy supported near $p_1\in \Scal$, we may assume this torus parametrises $v^{(1,3)}(p_1)\times {\beta'}_1^{(2)}$ in an orientation preserving manner. Assuming the radius $\varepsilon$ is small enough, we can act similarly around $\beta_1^{(1)}\times (p_1',p_1')$: remove another open solid torus (disjoint from the first one) and isotope near $p_1'$ to obtain a new boundary parametrising $\beta_1^{(1)}\times v^{(3,2)}(p_1')$ in an orientation preserving manner. (It is a bit subtle to see why it is with this and not with the inverse orientation. Removing a disc from the torus $\beta_{-1}^{(3)}\times {\beta'}_1^{(2)}$ yields a boundary $\bar{v}^{(3,2)}(p_1')$ because $\langle\beta_{-1},{\beta'}_1\rangle$ is negative. But then the canonical orientation of the boundary of $\partial(S^1\times X)$ is $S^1\times \overline{\partial X}$.) Denote by $\Sigma^{(1,3,2)}(\beta_1\times\beta_{-1}\times \beta_1')$ this $3$-torus with two open solid $2$-tori removed, as well with its embedding in $\conf_3(\Scal)$. See Figure \ref{fig:Ecal1-3torus}.

\input{fig_Ecal_parts}


We use auxiliary tubes and the surfaces from Section \ref{subsec:Bianchitrick} to cap these boundaries off; for this we must assume $g\ge 2$. 

In general, for any two elements $n_1,n_2$ of any set, let $\Sigma^{(n_1,n_2)}(\beta_i\times \beta_{-i})$ be the submanifold of $\conf_{\{n_1,n_2\}}(\Scal)$ where particles $1$ and $2$ are replaced by $n_1$ and $n_2$ respectively. Then, $\Sigma^{(1,3)}(\beta_2\times\beta_{-2})$ and $\beta_1'^{(2)}$ have disjoint supports in $\Scal$ and so we can take their product $\Sigma^{(1,3)}(\beta_2\times\beta_{-2})\times\beta_1'^{(2)}$, a submanifold of $\conf_3(\Scal)$ with one boundary torus parametrised by $v^{(1,3)}(p_2)\times \beta_1'^{(2)}$ in an orientation preserving manner. Similarly, the product $\beta_1^{(1)}\times \Sigma^{(3,2)}_{1,1}(\beta_{-2}\times\beta_2)$ has ones boundary torus parametrised by $\beta_1^{(1)}\times v^{(3,2)}(p_2)$. (That this is the orientation follows exactly as at the end of the previous paragraph). The latter two toric boundaries are isotopic to the boundaries of $\Sigma^{(1,3,2)}(\beta_1\times\beta_{-1}\times \beta_1')$. We glue them together with the help of tubes.

Our assumption that $\alpha_1$ intersects $\beta_1$ and $\beta_1'$ in this order implies that the path $\epsilon_{1,2}:p_1\rightsquigarrow p_2$ is disjoint from $\beta_1'$. Then the cylinder-circle product $v^{(1,3)}(\epsilon_{1,2})\times \beta_1'^{(2)}$ is an embedding of $[0,1]\times S^1\times S^1$ in $\conf_3(\Scal)$ with two boundary tori parametrising $v^{(1,3)}(p_2)\times \beta_1'^{(2)}$ and $\overline{v}^{(1,3)}(p_1)\times {\beta'_1}^{(2)}$ in an orientation preserving manner; this covers two out of our four boundaries. For the remaining two,  we ask for a new path $\epsilon_{1,2}':p_1'\rightsquigarrow p_2$ which intersects $U_{2g}$ transversely once at $\alpha_{-1}$ so that $\langle \epsilon_{1,2}', \alpha_{-1}\rangle =+1$ (this sign is determined by the cartography of the surface), and lies entirely in the subsurface $\Scal_{[1,2]}$; see Figure \ref{fig:inputforEcal}. Then the analogous circle-cylinder product  $\beta_1^{(1)}\times v^{(3,2)}(\epsilon_{1,2}')$ has boundaries $\beta_1^{(1)}\times v^{(3,2)}(p_1')$ and $\beta_1^{(1)}\times \overline{v}^{(3,2)}(p_2)$. (That the orientation reversal went to the endpoint rather than the startpoint is again from the formula $\partial(S^1\times X)=S^1\times \overline{\partial X}$.)

We can now glue the above five pieces along their boundaries to form a closed submanifold. A summary of them appears in Table \ref{tab:submanifolds}, along with whether the common boundaries are equally  or oppositely oriented. To obtain a canonically oriented manifold, we must glue in an orientation reversing way which leads us to reverse the orientations of three of our pieces. We must make sure that each pair of boundaries is glued in an orientation reversing manner: this leads us to reverse the orientation of three of the pieces and define
\begin{align}           
\Ecal=&\overline{\Sigma^{(1,3)}(\beta_2\times\beta_{-2})\times \beta_1'^{(2)}}
\cup\big(v^{(1,3)}(\epsilon_{1,2})\times \beta_1'^{(2)} \big)\cup \Sigma^{(1,3,2)}(\beta_1\times \beta_{-1}\times \beta_1') \label{eq:EcalDefn}
\\ &\hspace{10pt}\cup \overline{\big(\beta_1^{(1)}\times {v}^{(3,2)}(\epsilon_{1,2}')\big)} \cup\overline{\beta_1^{(1)}\times \Sigma^{(3,2)}(\beta_{-2}\times\beta_2)}, \notag
\end{align}
a closed, oriented, $3$-dimensional submanifold of $\conf_3(\Scal)$.

\begin{table}
    \centering
    \begin{tabular}{|c|c|c|c|}
    \hline
      & Submanifold   & Boundary 1 & Boundary 2 \\ \hline 
      1. & $\Sigma^{(1,3)}(\beta_2\times\beta_{-2})\times\beta_1'^{(2)}$ & $\emptyset$ & $v^{(1,3)}(p_2)\times \beta_1'^{(2)}$  \\
      & & & $\updownarrow +$\\
      2. & $v^{(1,3)}(\epsilon_{1,2})\times {\beta'_1}^{(2)}$  & $\overline{v}^{(1,3)}(p_1)\times {\beta'_1}^{(2)}$ & $v^{(1,3)}(p_2)\times \beta_1'^{(2)}$\\ 
      & & $\updownarrow -$ & \\
      3. & $\Sigma^{(1,3,2)}(\beta_1\times\beta_{-1}\times {\beta'_1})$   & $v^{(1,3)}(p_1)\times {\beta'_1}^{(2)}$ & $\beta_1^{(1)}\times v^{(3,2)}(p_1')$ \\ & & & $\updownarrow +$\\
      4. & $\beta_1^{(1)}\times {v}^{(3,2)}(\epsilon_{1,2}')$ & $\beta_1^{(1)}\times \bar{v}^{(3,2)}(p_2)$ & $\beta_1^{(1)}\times {v}^{(3,2)}(p_1')$ \\ 
      & & $\updownarrow-$ & \\ 5. & 
      $\beta_1^{(1)}\times \Sigma^{(3,2)}_{1,1}(\beta_{-2}\times\beta_2)$ & $\beta_1^{(1)}\times v^{(3,2)}(p_2)$ &  $\emptyset$ \\[5pt]
    \hline
    \end{tabular}
    \vspace{6pt}
    \caption{\footnotesize The five pieces of $\Ecal$ with their boundaries. The vertical arrows point between identical boundaries, and the sign indicates whether the two orientations are equal or reverse. For a gluing to give a canonically oriented manifold, the orientations must be reverse. This justifies why we invert the orientations of the first and last two manifolds in \eqref{eq:EcalDefn}, and therefore of their boundaries.}
    \label{tab:submanifolds}
\end{table}



\begin{proposition}\label{prop:Ecalintersections}
The submanifold $\Ecal$ of $\conf_3(\Scal)$ is supported in the subsurface $\Scal_{[1,2]}$ and intersects $\conf_3(U_{2g})$ transversally $5$ times in the components and with the signs detailed below:
\begin{align*}
    \langle \tri^{(1,2)}(\alpha_1) \tri^{(3)}(\alpha_{-1}),\Ecal\rangle&=+1:(1),\\
     \langle \tri^{(1)}(\alpha_1) \tri^{(2,3)}(\alpha_{-1}),\Ecal\rangle&=-1:(2),\\
     \langle \tri^{(1)}(\alpha_1) \tri^{(3,2)}(\alpha_{-1}),\Ecal\rangle&=-1:(3),\\
     \langle \tri^{(2)}(\alpha_1) \tri^{(1)}(\alpha_{2}) \tri^{(3)}(\alpha_{-2}),\Ecal\rangle&=+1:(4),\\
     \langle \tri^{(1)}(\alpha_{1})\tri^{(2)}(\alpha_2)  \tri^{(3)}(\alpha_{-2}),\Ecal\rangle&=-1:(5).
\end{align*}
\end{proposition}
\begin{proof}
    All the intermediate boundaries are supported away from $U_{2g}$ so we are free to consider the intersections of each piece separately.

    Firstly, the order that $\beta_1,\beta_1'$ intersect $\alpha_1$ gives a unique intersection of $\conf_3(U_{2g})$ with the torus $\Sigma^{(1,3,2)}(\beta_1\times \beta_{-1}\times \beta_1')$ in the component $\tri^{(1,2)}(\alpha_1)\times \tri^{(3)}(\alpha_{-1})$. Following Section \ref{sec:intersectionconventions}, the sign of the intersection is computed as 
    \begin{align*}
        \langle \alpha_1^{(1)}&\times\alpha_{1}^{(2)}\times \alpha_{-1}^{(3)},\beta_1^{(1)}\times\beta_{-1}^{(3)}\times\beta_1'^{(2)}\rangle\\
        &=(-1)^{2+2}\langle  \alpha_1,\beta_1\rangle\cdot \langle \alpha_{-1}, \beta_{-1}\rangle\cdot\langle \alpha_1, \beta_1'\rangle\\
        &=+1.
    \end{align*}

    For the circle-cylinder product $\overline{\beta_1'^{(1)}\times v^{(3,2)}(\epsilon_{1,2}')}$, Lemma \ref{lem:tubeintersection} tells us we have two intersections: one with each of $\tri^{(1)}(\alpha_1)\times \tri^{(3,2)}(\alpha_{-1})$ and $\tri^{(1)}(\alpha_1)\times \tri^{(2,3)}(\alpha_{-1})$. The former has sign
    \begin{align*}
        -\langle \tri^{(1)}(\alpha_1)&\times \tri^{(3,2)}(\alpha_{-1}), \beta_1'^{(1)}\times {v}^{(3,2)}(\epsilon_{1,2}')\rangle\\
        &= -(+1)\langle \tri^{(1)}(\alpha_1), \beta_1'^{(1)}\rangle\cdot \langle \tri^{(3,2)}(\alpha_{-1}),  {v}^{(3,2)}(\epsilon_{1,2}')\rangle\\
        &=-(+1)(+1)(+1)=-1
    \end{align*}
    where the external minus sign is due to the orientation reversal. Similarly the latter intersection has sign $-1$.

    Finally, the extremal piece $\overline{\Sigma^{(1,3)}(\beta_2\times\beta_{-2})\times \beta_1'^{(2)}}$  intersects only $\tri^{(2)}(\alpha_1)\times \tri^{(1)}(\alpha_{2})\times \tri^{(3)}(\alpha_{-2})$ transversally once, and the intersection has sign
    $$-\langle \alpha_1^{(2)}\times \alpha_2^{(1)}\times \alpha_{-2}^{(3)}, \beta_2^{(1)}\times \beta_{-2}^{(3)}\times \beta_1'^{(2)}\rangle,$$
    where the outside minus sign is due to the orientation reversal. The odd permutation contributes a further $-1$ but the pairwise intersections of arcs are all positive, so  overall we get $+1$. On the other hand, the unique intersection of $\overline{\beta_1^{(1)}\times \Sigma^{(3,2)}(\beta_{-2}\times\beta_2)}$  with $\tri^{(1)}(\alpha_{1})\tri^{(2)}(\alpha_2)  \tri^{(3)}(\alpha_{-2})$ has sign
    $$-\langle \alpha_1^{(1)}\times \alpha_2^{(2)}\times\alpha_{-2}^{(3)}, \beta_1^{(1)}\times \beta_{-2}^{(3)}\times \beta_2^{(2)}\rangle $$
    which is $-(+1)\cdot 1\cdot 1\cdot 1=-1$.
\end{proof}

We will want that $\Ecal$ annihilates $\tri^3(\Ical^3)\subset \Hcal_3(\Scal)$ under the intersection pairing. We achieve this by symmetrising it with respect to the transposition $\tau\in \Sfrak_3$ of  $1$ and $2$: define
$$\sEcal=\Ecal\sqcup \tau\Ecal$$ as the disjoint union of the two submanifolds in $\conf_3(\Scal)$.

\begin{proposition}\label{prop:sEcalintersections}
The submanifold $\sEcal$ of $\conf_3(\Scal)$ is supported in the subsurface $\Scal_{[1,2]}$ and intersects $\conf_3(U_{2g})$ transversally $10$ times. Four of these occur in the components $$\tri^{(1)}(\alpha_1)\tri^{(2)}(\alpha_2)\tri^{(3)}(\alpha_{-2}), \hspace{6pt} \tri^{(2)}(\alpha_1)\tri^{(1)}(\alpha_2)\tri^{(3)}(\alpha_{-2}),$$ and come in two pairs of opposite signs, and thus, vanishing algebraic contribution. The remaining six intersections have the following signs:
\begin{align*}
    \langle \tri^{(1,2)}(\alpha_1) \tri^{(3)}(\alpha_{-1}),\sEcal\rangle&=+1:(1),\\
    \langle \tri^{(2,1)}(\alpha_1) \tri^{(3)}(\alpha_{-1}),\sEcal\rangle&=+1:(1)',\\
     \langle \tri^{(1)}(\alpha_1) \tri^{(2,3)}(\alpha_{-1}),\sEcal\rangle&=-1:(2),\\
     \langle \tri^{(1)}(\alpha_1) \tri^{(3,2)}(\alpha_{-1}),\sEcal\rangle&=-1:(3),\\
     \langle \tri^{(2)}(\alpha_1) \tri^{(1,3)}(\alpha_{-1}),\sEcal\rangle&=-1:(2)',\\
     \langle \tri^{(2)}(\alpha_1) \tri^{(3,1)}(\alpha_{-1}),\sEcal\rangle&=-1:(3)'.
\end{align*}
The intersection of $\sEcal$ with $\tri^3(\Ical^3)$ vanishes.
\end{proposition}
\begin{proof}
    The permutation $\tau$ acts in an orientation preserving fashion on $\Scal^3$ and thus on $\conf_3(\Scal)$ (we are permuting $\Scal$-factors which have even dimension). The intersections of $\tau\Ecal$ with $\conf_3(U_{2g})$ can then be read off from Proposition \ref{prop:Ecalintersections} by swapping particles $1$ and $2$ throughout. The intersection $(4)$ of $\Ecal$ cancels with intersection $(5)'$ of $\tau\Ecal$, and vice versa. The remaining intersections are as appearing.

    For the last assertion, it suffices that $\langle \tri^{(1)}(\alpha_{i_1})\tri^{(2)}(\alpha_{i-2})\tri^{(3)}(\alpha_{i_3}), \sEcal\rangle$ vanishes for all $i_1,i_2,i_3\in \{\pm1, \ldots, \pm g\}$. So far this is automatically true in all except the cases $(i_1, i_2,i_3)=(1,1,-1)$, $(1,-1,-1)$, and $(-1,1,-1)$. Now recall the formula $\tri^{(i)}\times \tri^{(j)}=\tri^{(i,j)}-\tri^{(j,i)}$ of Lemma 2.1 of \cite{LooijengaStavrou25}. Subtracting $(1)'$ from $(1)$ we get the vanishing intersection for the first of the three cases; the subtractions $(2)-(3)$ and $(2)'-(3)'$ give the other two.
\end{proof}
\subsection{The perfect pairing}
We henceforth impose the assumption $g\ge n$. For each $\bfrak\in \BB^{s,r}_\nc$ with $r+s=n$, we construct the dual submanifold $\sEcal_\bfrak$ of $\conf_n(\Scal)$.

Write out $\bfrak=\mu^{s,r}_{\uk,\ul}(\mfrak)$. The chord diagram $(\uk,\ul)$ of length $r$ in $[s]$ determines a unique a non-decreasing surjection $[n]\to [s]$ with the property that each $k_j$, with $j=1,\ldots, r$, has two consecutive preimages; call these preimages $\kappa_j, \kappa_j+1$. Let $\lambda_j$ be the unique preimage of $l_j$ for each $j=1,\ldots, r$, and $\rho_1<\cdots< \rho_{s-2r}$ be the remaining elements of $[n]$. 

For each $j=1,\ldots, r$, let 
$$\sEcal_j=\sEcal^{(\kappa_j,\kappa_j+1,\lambda_j)}_{[2j-1,2j]}$$
be the closed, oriented, $3$-dimensional submanifold of $\conf_{\{\kappa_j,\kappa_j+1,\lambda_j\}}(\Scal)$ that replaces points $1,2,3$ of $\sEcal$ with $\kappa_j,\kappa_j+1,\lambda_j$, respectively, and the curves $\alpha_{\pm 1}, \alpha_{\pm 2}$ of $\sEcal$ with $\alpha_{\pm (2j-1)}, \alpha_{\pm 2j}$, respectively. Then each $\sEcal_j$ is supported in the subsurface $\Scal_{[2j-1,2j]}$, so that the  $\sEcal_1,\ldots, \sEcal_r$ have pairwise disjoint supports.

Proceeding, we construct a dual torus $\TT_\mfrak$ to $\mfrak$. Write the monomial out as $\mfrak=a_{i_1}\otimes \cdots \otimes a_{i_{n-2r}}\in \MM^{n-2r}_{+,n-2r}$ where the indices $i_1,\ldots, i_{n-2r}$ are in the range $>g-(n-2r)\ge 2r$ using the assumption $g\ge n$. Take pairwise disjoint curves $\eta_1,\ldots, \eta_{n-2r}$ that are dual to $\alpha_{i_1},\ldots, \alpha_{i_{n-2r}}$, respectively; ensure that each $\eta_j$, for $j=1,\ldots, s-2r$, is a parallel copy of $\beta_{i_j}$ lying entirely in the subsurace $\Scal_{[i_j]}$. 
Define the torus
$$\TT_\mfrak:=\eta_1^{(\rho_1)}\times \cdots \times \eta_{n-2r}^{(\rho_{n-2r})}$$
which is an $(n-2r)$-dimensional submanifold of $\conf_{\{\rho_1,\ldots, \rho_{n-2r}\}}(\Scal)$. This torus is supported in the subsurface $\Scal_{[>2r]}$ and thus has disjoint support from each of the $\sEcal_1,\ldots, \sEcal_r$. 

The disjointness of supports permits us to define the product $$\sEcal_\bfrak=\sEcal_1\times \cdots \times \sEcal_r\times \TT_\mfrak$$ 
as an oriented, $n$-dimensional submanifold of $\conf_n(\Scal)$.
The following is a consequence of  Propositions \ref{prop:toricintersections} and \ref{prop:Ecalintersections}.
\begin{proposition}
    The submanifold $\sEcal_\bfrak$ intersects $\conf_n(U_{2g})$ transversally at $10^j$ points, at each of which the particles $\kappa_j, \kappa_j+1,\lambda_j$ lie in the union of the $1$-cells $\alpha_{\pm(2j-1)},\alpha_{\pm 2j}$, for $j=1,\ldots, r$, (there are $5$ such choices for each $j$), and each particle $\rho_j$, for $j=1,\ldots, s-2r$, lies on the cell $\alpha_{i_j}$ at the intersection with $\eta_j$ (there is one such choice here). Each intersection involves exactly $r$ cells $\alpha_{i}$ with $i<0$, and these indices satisfy $|i|\le 2r$. 
    
    Furthermore, the intersection of $\sEcal_\bfrak$ with $\tri^n(\Ical^{k+1})$ vanishes.
\end{proposition}

We now prove that equation \eqref{eq:sEcalPerfectPairing} holds. This is a long proof, but it justifies all the choices in the construction of 
$\sEcal_\bfrak$.

\begin{theorem}\label{thm:sEcalPerfectPairing}
    Suppose $\bfrak,\bfrak'\in \cup_{s+r=n}\BB^{s,r}_\nc$. Then $\langle \tri^n(\widetilde{\bfrak}), \sEcal_{\bfrak'}\rangle=\pm 1$ if $\bfrak=\bfrak'$, and $=0$ otherwise.
\end{theorem}
\begin{proof}
    We first give an explicit expression for $\tri^n(\widetilde{\bfrak})$. 
    Write out $\bfrak=\mu^{s,r}_{\uk,\ul}(\mfrak)$ and $\mfrak=a_{i_1}\otimes \cdots \otimes a_{i_n}$. Then the decomposition formula in the form \eqref{eq:tri^ngeneraltilde} expands  $\tri^n(\widetilde{\bfrak})$ as 
    \begin{equation}\label{eq:tri^n(tilde(bfrak))}
        \sum_{\substack{\Nbold_1,\ldots, \Nbold_s\neq\emptyset\\ \Nbold_1\Nbold_2\cdots\Nbold_s=\underline{[n]}}}\hspace{-16pt}\pm\left(\prod_{j=1}^r\left(\sum_{1\le \pm i\le g} \sign(i)\tri^{\Nbold_{k_j}}(\alpha_i)\tri^{\Nbold_{l_j}}(\alpha_{-i})\right)\right)\times \tri^{\Nbold_{p_1}}(\alpha_{i_1}) \cdots  \tri^{\Nbold_{p_{s-2r}}}(\alpha_{i_{s-2r}}).
    \end{equation}
    Here, we have used that $\mu^{s,r}_{\uk,\ul}$ inserts $\mu$ in the $r$ pairs of tensor slots prescribed by $(\uk,\ul)$, and the expression of $\tri^{n}(\widetilde{\mu})$ from equation \eqref{eq:tri^nmutilde}. We have permuted the factors so that all the $\widetilde{\mu}$ come first, at the expense of a sign $\pm$. The outer summation is over all partitions-with-orders $\Nbold_1,\ldots, \Nbold_s$ with non-empty parts and so that the concatenated ordering $\Nbold_1\Nbold_2\cdots\Nbold_s$ is the ordering $1<2<\cdots<n$; and $p_1<p_2<\cdots<p_{s-2r}$ are the elements of $[s]$ not in the chord diagram $(\uk,\ul)$.

    We simplify further. Suppose in a given partition-with-orders $\Nbold_1,\ldots, \Nbold_s$, there is $j\in \{1,\ldots ,r\}$ for which both $N_{k_j}$ and $N_{l_j}$  are singletons; say their unique elements are $u,v$, respectively.  Then the factor
    $$\sum_{1\le \pm i\le g} \sign(i)\tri^{\Nbold_{k_j}}(\alpha_i)\tri^{\Nbold_{l_j}}(\alpha_{-i})$$
    is equal $\tri^{(u,v)}(\widetilde{\mu})=\tri^{(u,v)}(\zeta)=0\in \Hcal_{\{u,v\}}(\Scal)$ and thus makes the entire contribution  of this partition-with-orders vanishes. We may then restrict the outer summation of \eqref{eq:tri^n(tilde(bfrak))} to include only those partitions with $|N_{k_j}|+|N_{l_j}|\ge 3$ for all $j=1,\ldots, r$. As $n=s+r$, and all parts $N_1,\ldots, N_r$ are non-empty, the last inequality is forced to the equality $|N_{k_j}|+|N_{l_j}|= 3$ and also that each $N_{p_j}$, for $j=1,\ldots, s-2r$, is a singleton. Distributing out the expression \eqref{eq:tri^n(tilde(bfrak))}, we then get that $\tri^n(\widetilde{\bfrak})$ is a sum of terms
    \begin{equation}\label{eq:monomialsoftri^n(bfrak)}
    \pm\tri^{\Nbold_{k_1}}(\alpha_{j_1}) \tri^{\Nbold_{l_1}}(\alpha_{-j_1})\cdots  \tri^{\Nbold_{k_r}}(\alpha_{j_r}) \tri^{\Nbold_{l_r}}(\alpha_{-j_r}) \tri^{(r_1)}(\alpha_{i_1}) \cdots  \tri^{(r_{s-2r})}(\alpha_{i_{s-2r}})
    \end{equation}
    where each $j_1,\ldots, j_r\in \{\pm 1,\ldots, \pm g\}$ and the $r_1<\cdots< r_{s-2r}\in [n]$ are the elements of $N_{p_1},\ldots,N_{s-2r}$, respectively. 
    
    From now on, we fix a product $\dfrak$ of type \eqref{eq:monomialsoftri^n(bfrak)} and consider the subspace it defines in $\conf_n(U_{2g})$. Let us take another $\bfrak'=\mu^{s',r'}_{\uk',\ul'}(\mfrak')\in \cup_{s+r=n}\BB^{s,r}_\nc$ and assume that the subspace $\dfrak$ contains an intersection of $\conf_n(U_{2g})$ with $\sEcal_{\bfrak'}$. We will deduce that this occurs only if $\bfrak=\bfrak'$ and, furthermore, only if $\dfrak$ is of a prescribed form.

    Since $i_1,\ldots, i_{s-2r}>0$ in \eqref{eq:monomialsoftri^n(bfrak)}, then the monomial $\dfrak$ involves at most $r$ different $\alpha_{j}$ with $i<0$ (the ``at most'' is because the $|j_1|,\cdots, |j_r|$ might not be distinct). On the other hand each intersection point $\sEcal_{\bfrak'}\cap \conf_n(U_{2g})$ involves exactly $r'$ negative $\alpha_j$s, all of which have indices $|j|\le 2r'$. As one of these intersections is in $\dfrak$, we must have $r\ge r'$ and all the negative indices of $\dfrak$ must satisfy $|j_1|,\ldots, |j_r|\le 2r$. On the other hand, $\dfrak$ has at least $|N_{k_1}|+\cdots+|N_{l_r}|=3r$ of the $n$ particles in the first $4r$ cells $\alpha_{\pm 1},\ldots, \alpha_{\pm 2r}$ (the ``at least'' is because some of the ${i_1},\ldots, {i_{s-2r}}$ might also be $\le 2r$), while all intersection points in $\sEcal_{\bfrak'}\cap \conf_n(U_{2g})$ have exactly $3r'$ points in these cells; this shows  $3r'\ge 3r$, and in all we have $r=r'$. Recalling that $s+r=n=s'+r'$, we also have $s=s'$. 

    We have assumed that the indices of the monomial $\mfrak$ satisfy ${i_1},\ldots, {i_{s-2r}}> g-(s-2r)$ and the latter is $=2r+g-s\ge 2r+g-n\ge 2r$, using $g\ge n$. This means that these indices are all distinct from any of the ${\pm j_1},\ldots, {\pm j_r}$.
    From the types of possible intersections of $\sEcal$ with $\conf_n(U_{2g})$ from Proposition \ref{prop:sEcalintersections} only the types $(1)$ and $(2)$ can occur in $\dfrak$, as other intersections either involve  a factor $\tri^{(i,j)}$ where $j\neq i+1$, or use more cells than they should. So each factor $\sEcal^{(\kappa_j',\kappa_j+1',\lambda_j')}_{[2j-1,2j]}$ of $\sEcal_{\bfrak'}$ intersect one of the factors $\tri^{\Nbold_{k_t}}(\alpha_{j_t}) \tri^{\Nbold_{l_t}}(\alpha_{-j_t})$ of $\dfrak$, which must then, by the same proposition, equal one of the following terms:
    \begin{align*}
        \tri^{(\kappa_j',\kappa_j'+1)}(\alpha_{2j-1})\tri^{(\lambda_j')}(\alpha_{-(2j-1)}) &\hspace{6pt}(1),\\
        \tri^{(\kappa_j')}(\alpha_{2j-1})\tri^{(\kappa_j'+1,\lambda_j')}(\alpha_{-(2j-1)}) &\hspace{6pt}(2),
    \end{align*}
    A consequence of this is that the $|j_1|,\ldots, |j_r|$ are pairwise distinct and in bijection with the odd numbers $1,3,\ldots,2r-1$ in some order; given $i\in \{1,\ldots, r\}$, let $|j_{i}|=2t(i)-1$ for some $t(i)\in \{1,\ldots, r\}$. Recall that the concatenated order $\Nbold_1\cdots\Nbold_s$ was required to be the order $1<\cdots<n$, that each part is non-empty, and that each pair $(k_{1},l_{1}),\ldots, (k_{r},l_{r})$ is non-consecutive. Then the pair $(\Nbold_{k_{t}},\Nbold_{l_{t}})$ is either $((u,u+1),(v))$ for some $u+2<v$ in $[n]$, or $((u),(v,v+1))$ for some $u+1<v$. Then, only option $(1)$ is possible: so $\Nbold_{k_{j_{t(j)}}}=(\kappa_j',\kappa_j'+1)$ and $\Nbold_{l_{j_{t(j)}}}=(\lambda_j')$. Furthermore as $\kappa_1'<\cdots<\kappa_r'$, but also $k_1<\cdots<k_r$ which means $\Nbold_{k_1}<\cdots<\Nbold_{k_r}$, it follows that $t(i)=i$ for all $i$. Then $(\uk,\ul)=(\uk',\ul')$ and $j_i=2i-1$ for all $i=1,\ldots, r$.

    Finally, the increasing order of the $r_1,\ldots, r_{s-2r}$ and the $\rho_1, \ldots, \rho_{s-2r}$ and the fact that they are the remaining elements of $[n]$ implies that $r_1=\rho_1,\ldots,\allowbreak r_{s-2r}=\rho_{s-2r}$. Then the $(s-2r)$-cube  $\tri^{(\rho_1)}(\alpha_{i_1}) \cdots  \tri^{(\rho_{s-2r})}(\alpha_{i_{s-2r}})$ must contain the unique intersection of $\TT_{\mfrak'}$ with $\conf_{r_1,\ldots, r_{s-2r}\}}(U_{2g})$, so each $\alpha_{i_j}$ must be dual to the curve $\eta_j'$ which in turn is dual to $\alpha_{i_j'}$. Then $\alpha_{i_j}=\alpha_{i_j'}$ for all $j=1,\ldots, s-2r$ and we have also proved that $\mfrak=\mfrak'$.

    We have proved that if $\langle \tri^n(\widetilde{\bfrak}, \sEcal_{\bfrak'}\rangle \neq 0$, then $\bfrak=\bfrak'$. We have also proved that in the case $\bfrak=\bfrak'$, then $\langle \dfrak, \sEcal_{\bfrak'}\rangle\neq 0$ only for the summand    
    $$ \dfrak=\pm\left(\prod_{j=1}^r\tri^{(\kappa_j,\kappa_j+1)}(\alpha_{2j-1}) \tri^{(\lambda_j)}(\alpha_{-(2j-1)})\right)\times  \tri^{(\rho_1)}(\alpha_{i_1}) \cdots  \tri^{(\rho_{s-2r})}(\alpha_{i_{s-2r}})$$ 
    of $\tri^n(\widetilde{\bfrak})$. Finally, this latter intersection between subspaces of $\conf_n(\Scal)$ is transversal at a unique point. Therefore it is $\pm 1$.
\end{proof}
\begin{remark}[On the genus condition] The condition $g\ge n$ is necessary in the proof. For example, if $n$ is even, then the image of the $\mu^{n,n/2}_{\uk,\ul}$, running over all chord diagrams, are not linearly independent if $g<n$. That does not impede Theorem \ref{thm:sEcalPerfectPairing} to hold via another proof for $g<n$, however.
\end{remark}
\begin{remark}[The non-consecutive condition]
    The proof would fail if some pair $(k_i,l_i)$ were consecutive; let us look as an example at the case $n=3$, and $\bfrak=\mu^{2,1}_{(1),(2)}(1)=\mu\in H^{\otimes 2}$, so that $\sEcal_\bfrak=\sEcal$ in $\conf_3(\Scal)$.
    Then $$\tri^3(\widetilde{\mu})=\sum_{1\le \pm i\le g} \sign(i)\left(\tri^{(1,2)}(\alpha_i)\tri^{(3)}(\alpha_{-i})+\tri^{(1)}(\alpha_i)\tri^{(2,3)}(\alpha_{-i})\right),$$
    with only the summands $\tri^{(1,2)}(\alpha_1)\tri^{(3)}(\alpha_{-1})$ and $\tri^{(1)}(\alpha_1)\tri^{(2,3)}(\alpha_{-1})$
    intersecting $\sEcal$; the signs are $+1$ and $-1$, respectively, coming from parts $(1)$ and $(2)$ of Proposition \ref{prop:Ecalintersections}. Thus $\langle \tri^3(\widetilde{\bfrak}),\sEcal_\bfrak\rangle =0$.
\end{remark}

    


Theorem \ref{thm:independence_r+s=n} is now a simple corollary.
\begin{proof}[Proof of Theorem \ref{thm:independence_r+s=n}]
    The dual set $\{\langle -, \sEcal_\bfrak\rangle:\bfrak\in \cup_{s+r=n}\BB^{s,r}_\nc\}$ is a set of functionals vanishing on $\tri^{n}(\Ical^{s+1})$ and thus defined on $\Hcal_n(\Scal)/\tri^{n}(\Ical^{s+1})$ where they form a partial dual basis for the desired set.
\end{proof}

\section{The configuration Johnson kernels}\label{sec:configurationkernels}
In this section, we translate the results of Section \ref{sec:keroftri^n} to results about the quotient $J^\cfg_{g,*}(n)/J_{g,*}(n)$. We fix $n\ge 1$ throughout. We also introduce the following notation for a filtered vector space: if we have $V=F_0V\supset F_1V\supset F_2V\supset\cdots$, then let us denote $$V|^i_j=F_iV/F_{j+1}V$$
for the relevant quotients.
If $i$ is omitted then it is assumed to be $0$.

\subsection{Configuration Johnson homomorphisms}
We recall that $J^\cfg_{g,*}(n)$ is the kernel of the action $\Gamma_{g,*}\curvearrowright \Ical^\cfg_n$.
Filtering, as before, $\Ical^\cfg_n$ by $F_i\Ical^\cfg_n=\tri^n(\Ical^i)$, we will now consider the more general kernels
$$J^\cfg_{g,*}(n,k)=\ker(\Gamma_{g,*}\curvearrowright \Ical^\cfg_n|_k).$$
Here, of course, $J^\cfg_{g,*}(n,n)=J^\cfg_{g,*}(n)$. 
Supposing $k\ge 1$, then the action of any 
$\phi\in J^\cfg_{g,*}(n,k)$ on $\Ical^\cfg_n|_{k+1}$ has the property that $\tau_\phi:=\phi-\id$ lands in $\Ical^\cfg_n|^{k+1}_{k+1}=\gr^\Ical_{k+1}\Ical^\cfg_n$, so we obtain a function \begin{equation}\label{eq:pretau}
    \phi\in J^\cfg_{g,*}(n,k)\longmapsto \tau_\phi=(\phi-\id)\in \Hom(\Ical^\cfg_n|_{k}, \Ical^\cfg_n|^{k+1}_{k+1}).
\end{equation}
\begin{lemma}
    The function \eqref{eq:pretau} is a group homomorphism, whose kernel is $J^\cfg_{g,*}(n,k+1)$, and it factors through the injection
    $$\tau^\cfg_{g,*}(n,k):J^\cfg_{g,*}(n,k)/J^\cfg_{g,*}(n,k+1)\hookrightarrow \Hom(H,\Ical^\cfg_n|^{k+1}_{k+1}).$$
    It is equivariant with the conjugation action of $\Gamma_{g,*}$ on the domain and the natural action on the codomain. 
    
    In particular, $J^\cfg_{g,*}(n,k)/J^\cfg_{g,*}(n,k+1)$ is a finitely generated abelian group, and its action of $\Gamma_{g,*}$ factors through $\Sp_{2g}(\ZZ)$. After tensoring with $\QQ$, it is an algebraic representation of weight $\le k+2$.
\end{lemma}
\begin{proof}
    The surjection $\tri^n$ makes $\Ical^\cfg_n$ into a filtered ring (without a unit), whose associated graded ring $\gr^\Ical_\pt\Ical^\cfg_n$ is generated in degree $1$ as that is true for $\gr^\Ical_\pt\QQ\pi$ by Labute. Since $k\ge 1$, $J^\cfg_{g,*}(n,k)$ acts trivially on $\gr^\Ical_1\Ical^\cfg_n=H$ and thus on the whole $\gr^\Ical_\pt\Ical^\cfg_n$, and, in particular, on $\Ical^\cfg_n|^{k+1}_{k+1}$. Then if $\phi\in J^\cfg_{g,*}(n,k)$ and $x\in \Ical^\cfg_n|^{k+1}_{k+1}$, we get \begin{align*}
        \tau_{\phi\circ\psi}(x)&=(\phi\circ\psi)(x)-x=\phi(\psi(x))-x=\phi(x+\tau_\psi(x))-x\\
        &=x+\tau_\phi(x)+\tau_\psi(x)-x=(\tau_\phi+\tau_\psi)(x)
    \end{align*}
    since the action of $\phi$ on $\tau_\psi(x)\in \Ical^\cfg_n|^{k+1}_{k+1}$ is trivial. This proves that \eqref{eq:pretau} is a homomorphism. 
    To be in the kernel of \eqref{eq:pretau} is equivalent $\tau^\cfg_{g,*}(n,k)$ is equivalent to acting trivially on $\Ical^\cfg_n|^1_{k+1}$ and thus lying in $J^\cfg_{g,*}(n,k+1)$, so \eqref{eq:pretau} factors through $J^\cfg_{g,*}(n,k)/J^\cfg_{g,*}(n,k+1)$. Finally, the multiplicativity of the filtration implies that $\phi\in J^\cfg_{g,*}(n,k)$ also acts trivially on $\Ical^\cfg_n|^2_{k+1}$ and hence the homomorphism $\tau_\phi$ vanishes $\Ical^\cfg_n|^2_{k+1}$ and thus factors through $\Ical^\cfg_n|^1_1=H$. Then \eqref{eq:pretau} factors through the claimed monomorphism. 
    
    It is a trivial exercise to check that the equivariance with the $\Gamma_{g,*}$ actions. The last properties on the domain follow because  clearly makes the map equivariant; the action factors through $\Sp_{2g}(\ZZ)$ because $\Hom(H,\Ical^\cfg_n|^{k+1}_{k+1})$ is a $\QQ$-vector space, and thus torsion free, and as an $\Sp_{2g}(\ZZ)$-representation it is a quotient of $H\otimes H^{\otimes k+1}$.
\end{proof}

The discussion giving rise to $\tau^\cfg_{g,*}(n,k)$ is the exact analogue to that of the classical Johnson homomorphisms
\begin{equation}
    \tau_{g,*}(k):J_{g,*}(k)/J_{g,*}(k+1)\to \Hom(H,\QQ\pi|^{k+1}_{k+1}).
\end{equation}
We remark that albeit working over $\QQ$ instead of $\ZZ$, the map $\tau_{g,*}(k)$ remains injective as $\pi/\pi^{(n+1)}$ is free abelian by Labute \cite{Labute} and thus embeds in $\QQ\pi|_{k+1}$. Since $\QQ\pi|_{k}$ surjects on $\Ical^\cfg_n|_k$ we have the inclusions $J_{g,*}(k)\subset J_{g,*}^\cfg(n,k)$, for $k\ge n$, which induce the comparison comparison maps $$\comp_k:J_{g,*}(k)/J_{g,*}(k+1)\to J^\cfg_{g,*}(n,k)/J^\cfg_{g,*}(n,k+1).$$ These sit in the commuting squarea
\begin{equation}\label{eq:comparisonmapofJohnsons}
 \begin{tikzcd}
    J_{g,*}(k)/J_{g,*}(k+1)\rar["\comp_k"]\dar["\tau_{g,*}(k)", hook]& J^{\cfg}_{g,*}(n,k)/J^{\cfg}_{g,*}(n,k+1) \dar["\tau^{\cfg}_{g,*}{(n,k)}", hook]\\
        \Hom(H,\QQ\pi|^{k+1}_{k+1})\rar[two heads, "\tri_k"] & \Hom(H,\Ical^\cfg_n|^{k+1}_{k+1}),
    \end{tikzcd}
\end{equation}
with the bottom surjection coming from $\gr^\Ical_{k+1}\tri^{k+1}$.
We will determine the kernel and cokernel of each $\comp_k$.

\subsection{The kernel of $\comp_k$}
It will be convenient to work with the Johnson homomorphisms 
\begin{equation}
    \tau_{g,1}(k):J_{g,1}(k)/J_{g,1}(k+1)\to \Hom(H,\QQ\pi_1|^{k+1}_{k+1})
\end{equation}
for $\Gamma_{g,1}$ as well. The target can be simplified by the chain of $\Sp_{2g}(\ZZ)$-equivariant isomorphisms
\begin{equation}\label{eq:selfdualitysimplification}
    \Hom(H,\QQ\pi_1|^{k+1}_{k+1})\cong H^{\vee}\otimes H^{\otimes k+1}\cong H^{\otimes k+2},
\end{equation}
where the first isomorphism is due to Fox \cite{Fox}. The last term has an action of the cyclic group $\Ccal_{k+2}$ of order $k+2$ by permuting the tensor factors. We have the following known result, e.g. see Morita's survey \cite[Prop~4.6]{Morita98survey}.

\begin{lemma}
    The image of $\tau_{g,1}(k)$ lies, under the isomorphism \eqref{eq:selfdualitysimplification}, in the $\Ccal_{k+2}$-invariant subspace $(H^{\otimes k+2})^{\Ccal_{k+2}}$.
    
\end{lemma} 
On the other hand, the surjection  $\Gamma_{g,1}\twoheadrightarrow\Gamma_{g,*}$ induces a comparison of Johnson homomorphisms
\begin{equation}\label{eq:comparisonmapofJohnsonsClassical}
 \begin{tikzcd}
    J_{g,1}(k)/J_{g,1}(k+1)\rar\dar["\tau_{g,1}(k)", hook]& J_{g,*}(k)/J_{g,*}(k+1)\dar["\tau_{g,*}{(k)}", hook]\\
    \Hom(H,\QQ\pi_1|^{k+1}_{k+1})\rar[two heads] & \Hom(H,\QQ\pi|^{k+1}_{k+1}).
    \end{tikzcd}
\end{equation}
The following, is due to Hain \cite{HainInfinitesimal}, in  a formulation from \cite[\S~6.1]{Morita98survey}.
\begin{proposition}\label{prop:Hain'sTheorem}
    The top map of diagram \eqref{eq:comparisonmapofJohnsonsClassical} is injective for all $k\neq 2$; for $k=2$ the kernel is $\ZZ$ generated by the boundary Dehn twist. After tensoring with $\QQ$ this map becomes surjective for all $k$.
\end{proposition}

Theorem \ref{thm:assgradedIcalcfg} puts the subspace $$\Hom(H,(\QQ\pi|^{k+1}_{k+1})^{\le 3(k+1)-2(n+1)})$$ in the kernel of the map $\tri_k$ from diagram \eqref{eq:comparisonmapofJohnsons}. 
With the next lemma we will be able to find an upper bound on the weight of this kernel. The proof is a careful analysis of what types of monomials can appear in these intersections. 

\begin{lemma}\label{lem:trivialintersectionofreps}
For any weight bound $-1\le w\le k$, the intersection of the subspaces \begin{equation}\label{eq:intersection1}
    (H^{\otimes k+2})^{C_{k+2}}\cap H\otimes (H^{\otimes k+1})^{\le w}\subset H^{\otimes k+2}
\end{equation} is of weight $\le w-1$.

If, furthermore, $g\ge k+2$, then the intersection
    \begin{equation}\label{eq:intersection2}
        \im \tau_{g,1}(k)\otimes \QQ\cap \left(H\otimes (H^{\otimes k+1})^{\le w}+ H\otimes \sum_{i=1}^{k}\im\mu_{i,i+1}\right) \subset H^{\otimes k+2}  
    \end{equation}
    is also of weight $\le w-1$.
\end{lemma}
\begin{remark}
    These intersections are obviously of weight $\le w-1$, so all we need to prove is that no irreducibles of weight $w+1$ appear.
\end{remark}
\begin{proof}[Proof of Lemma \ref{lem:trivialintersectionofreps}] For $w=-1$, there is nothing to prove so we assume $w\ge 0$. Also assume that $w$ is of the same parity as $k+1$; otherwise replace $w$ by $w-1$ and the meaning of the statements does not change.

The argument in both cases starts in the same way. Assume $V$ is an irreducible summand of weight $\wt(V)\ge w+1$ in the given intersection. If we write $\wt(V)=k+2-2r$ (we can do so from the parity assumption), then Proposition \ref{prop:weight(n-2r)ofH^n} gives a non-zero element $v\in V$ that is a linear combination of $\mu^{k+2,r}_{\uk,\ul}(\vfrak)$ for various chord diagrams $(\uk,\ul)$ and a fixed element $\mu^{\otimes r}\otimes \vfrak$ with $\vfrak\in \HH_{+,k+2-2r}^{k+2-2r}$ (recall Definition \ref{def:positivesubspace}). In particular, $v$ is of type $(k+2-r,r)$ according to Section \ref{sec:SPgenerators}, or simply \textit{$r$-negative}. Then $v$ is uniquely expressible as \begin{equation}\label{eq:firstfactordecompv}
    \sum_{1\le i\le g}a_i\otimes v_i+a_{-i}\otimes v_{-i}
\end{equation}
with each $v_{\pm i}\in H^{\otimes k+1}$. Clearly, each $v_{i}$ is $r$-negative, and each $v_{-i}$ is $r-1$-negative. The assumption $v\neq 0$ implies that not all $v_{\pm i}$ can vanish at once. We proceed to analyse the two cases separately.

\smallskip
\emph{Case: intersection \eqref{eq:intersection1}.} We are assuming on the one hand that $v\in H\otimes (H^{\otimes k+1})^{\le w}$, or equivalently that $v_{\pm i}\in (H^{\otimes k+1})^{\le w}$ for all $i$. The latter subspace of $H^{\otimes k+1}$ is $\QQ$-spanned at least $\ge (k+1-w)/2$ negative monomials; this is because all irreducible summands of weight $\le w$ involve at least $\ge (k+1-w)/2$ factors of $\mu$, each contributing for certain a negative $a_i$ factor. If some $v_i\neq 0$ for $i>0$, then $r\ge (k+1-w)/2$ so $\wt(V)=k+2-2r\le w+1$. Similarly, if some $v_{-i}\neq 0$ for $i>0$, then $r-1\ge (k+1-w)/2$ and $\wt(V)\le w-1$. Since we assumed $\wt(V)\ge w+1$ and $v\neq 0$, it follows that $v_{-i}=0$ for all $i>0$ and $r=(k+1-w)/2$. In particular $r\ge 1$ and $v=\sum_{1\le i\le g}a_i\otimes v_i$ is at least $1$-negative.

On the other hand, since $v$ is invariant under the cyclic $\Ccal_{k+2}$-action, then for any monomial appearing non-trivially in $v$, all the cyclic permutations of that monomial appear with the same non-trivial coefficient. But every monomial with positive factor in the first slot and at least one negative factor thereafter has a cyclic permutation with a negative factor in the first slot. We have proven $v$ contains monomials of the former type but none of the latter, giving us a contradiction, thus proving the first assertion of the statement.

\smallskip
\emph{Case: intersection \eqref{eq:intersection2}.}
First, let us observe that $V$ cannot entirely lie in the intersection $\im\tau_{g,1}(k)\otimes \QQ\cap H\otimes \sum_{i=1}^{k}\im\mu_{i,i+1}$; if it did, then $V\subset J_{g,1}(k)/J_{g,1}(k+1)\otimes \QQ$ would lie in the kernel of $J_{g,1}(k)/J_{g,1}(k+1)\to J_{g,*}(k)/J_{g,*}(k+1)$. By Proposition \ref{prop:Hain'sTheorem}, this kernel is trivial unless $k=2$ in which case it is the trivial representation $\QQ$ of weight $0$, while $V$ has weight $\ge w+1\ge 1$.

Now, write $v=u+u'$ for some $u\in H\otimes (H^{\otimes k+1})^{\le w}$ and $u'\in H\otimes \sum_{i=1}^{k}\im\mu_{i,i+1}$, assuming $u$ is not contained in the latter space. So we know  from the previous case that $u$ is non-zero and in the span of monomials that are at least $(k+1-w)/2$-negative. This, as before, implies $r\ge (k+1-w)/2$ or, equivalently, $\wt(V)\le w+1$. Our assumption $\wt(V)\ge w+1$ forces the weight of $V$ to be precisely $w+1$. It also forces that, in a decomposition of $u$ analogous to \eqref{eq:firstfactordecompv}, that all $u_i=0$ for $i<0$; in other words, any monomial summand of $u$ has its negative factors in the last $k+1$ tensor slots.

The assumption $g\ge k+2$ means that the fix element $\vfrak\in \HH^{k+2-2r}_{+,k+2-2r}$ involves only the generators $a_i$ for $i\ge 2r$. Then, each $\mu^{k+2,r}_{\uk,\ul}(\vfrak)$ contains, as summands, a permutation of $a_1\otimes a_{-1}\otimes \cdots \otimes a_{r}\otimes a_{-r}\otimes \vfrak$ that has the factors $a_1,a_{-1},\ldots,a_r,a_{-r}$ in the tensor slots $k_1,l_1,\ldots, k_r,l_r$, respectively; in the other tensor slots only terms involving $a_i$ with $i>r$ appear. It follows that the different $\mu^{k+2,r}_{\uk,\ul}(\vfrak)$ contain terms that cannot cancel between them and are thus linearly independent. In particular, there is a well-defined set of chord diagrams involved non-trivially in the linear combination $v$.

By virtue of lying in the span of the sum $H\otimes (H^{\otimes k+1})^{\le w}+H\otimes\sum\im\mu_{i,i+1}$, these chord-diagrams can only be of the following two types: (I) having at least one consecutive chord in the last $k+1$ slots (coming from some $\mu_{i,i+1}$), and (II) non-consecutive, and involving only the last $k+1$ slots. The $\Ccal_{k+2}$ invariance of $v$, along with the argument of the previous paragraph, enforces that this set of chord-diagrams is closed under cyclic permutation. But a chord diagram of type (II) of length $r\ge 1$ can be cycled to be still non-consecutive and to now involve the first slot; this is of neither type (I) or (II). Therefore $v$ must involve only type (I) chord diagrams which contradicts that $v\not\in H\otimes \sum_{i=1}^{k}\im\mu_{i,i+1}$.
\end{proof}




Suppose $A$ is a finitely generated free abelian group with an action of $\Sp_{2g}(\ZZ)$. In case $A\otimes \QQ$ is algebraic, we denote by $A^{\le w}$ the intersection of $A\subset A\otimes \QQ$ with $(A\otimes \QQ)^{\le w}$; this is a primitive subspace of $A$.

\begin{theorem}
    Assume $g\ge k+2$. Then for any $-1\le w\le k$, the kernel of the composition $$J_{g,*}(k)/J_{g,*}(k+1)\xrightarrow{\tau_{g,*}(k)}\Hom(H,\QQ\pi|^{k+1}_{k+1})\twoheadrightarrow\Hom(H,\QQ\pi|^{k+1}_{k+1}/(\QQ\pi|^{k+1}_{k+1})^{\le w})$$
    is the weight $\le w-1$ subgroup $(J_{g,*}(k)/J_{g,*}(k+1))^{\le w-1}$.
\end{theorem}
\begin{proof}
Let $V$ be an irreducible summand in $\ker\comp_k\otimes \QQ$ of weight $\ge w$. Then by Proposition \ref{prop:Hain'sTheorem} $V$ into a summand of $\im \tau_{g,1}(k)\otimes \QQ$, intersected with the kernel of the composition of the bottom row which is
$$H\otimes (H^{\otimes k+1})^{\le w}+ H\otimes \sum_{i=1}^{k}\im\mu_{i,i+1},$$
contradicting the second part of Lemma \ref{lem:trivialintersectionofreps}.
Suppose, on the other hand, $V$ is any irreducible summand of weight $\le w-1$ in $J_{g,*}(k)/J_{g,*}(k+1)\otimes \QQ$. Then the restriction of this composition on $V$ can be rewritten as an evaluation map
    $$ev: V\otimes H\to \QQ\pi|^{k+1}_{k+1}/(\QQ\pi|^{k+1}_{k+1})^{\le w}.$$
    Now the domain of $ev$ has weight $\le (w-1)+1=w$ whereas the codomain has no irreducibles of weight $\le w$. Thus $ev$ vanishes on $V\otimes H$ and thus the composition has $V$ in its kernel.
\end{proof}

Combining the latter theorem with our description of $\Ical^\cfg_n|^{k+1}_{k+1}$ the following is immediate.
\begin{theorem}\label{thm:differenceofJohnsonimages}
    If $k\le n$, then the kernel of the map $$\comp_k:J_{g,*}(k)/J_{g,*}(k+1)\to J^\cfg_{g,*}(n,k)/J^\cfg_{g,*}(n,k+1)$$ contains the subgroup $(J_{g,*}(k)/J_{g,*}(k+1))^{\le 3(k+1)-2(n+1)-1}$. If furthermore $g\ge \max(n,k+2)$, then this containment is an equality.
\end{theorem}
\subsection{Finite cokernel}
Since the codomain of $\comp_k$ is a finitely generated abelian group, it suffices to prove that $\comp_k$ is surjective after tensoring with $\QQ$. We start with an auxiliary statement.

\begin{lemma}\label{lem:extensionbyirrelevantrep}
    Suppose $1\to F\to E \to V\to 0$ is a $\Gamma_{g,*}$-equivariant short exact sequence of groups equipped with actions of $\Gamma_{g,*}$. Let $f:E\to U$ be a $\Gamma_{g,*}$-equivariant group homomorphism. Suppose also that $U,V$ are abelian, $V$ is finitely generated, and $\Hom_{\Gamma_{g,*}}(V\otimes \QQ, U\otimes \QQ)=0$. Then the image of the restriction $f|_F$ is finite index in the image of $f$.
\end{lemma}
\begin{proof}
    The map $f$ induces the $\Gamma_{g,*}$-equivariant map $\bar{f}: V\to U/\im(f|_F)$; by the trivial $\Hom$ condition, the latter map has torsion image; it is finite since $V$ is finitely generated. This image is $\im(f)/\im(f|_F)$.
\end{proof}

\begin{theorem}\label{thm:finitecokernel}
The map $\comp_k$ has finite cokernel.    
\end{theorem}
\begin{proof}
The map $\comp_k$ factors as the composition 
\begin{equation}\label{eq:comp_kfactoring}
\begin{tikzcd}[column sep=small]
        J_{g,*}(k)/J_{g,*}(k+1)\rar[hook] & J^\cfg_{g,*}(n,k)/J_{g,*}(k+1)\rar[two heads] & J^\cfg_{g,*}(n,k)/J^\cfg_{g,*}(n,k+1),
\end{tikzcd}
\end{equation}
where the inclusion, in its turn, factors through the sequence of inclusions 
\begin{align}
    (J_{g,*}(k)\cap J^\cfg_{g,*}(n,k))/J_{g,*}(k+1)&\subset (J_{g,*}(k-1)\cap J^\cfg_{g,*}(n,k))/J_{g,*}(k+1) \label{eq:filtrationbyintersectingwithJohnson}\\
    &\subset (J_{g,*}(k-2)\cap J^\cfg_{g,*}(n,k))/J_{g,*}(k+1) \notag \\
    &\subset \cdots \notag \\
    &\subset (J_{g,*}(1)\cap J^\cfg_{g,*}(n,k))/J_{g,*}(k+1) \notag \\
    &=J^\cfg_{g,*}(n,k)/J_{g,*}(k+1).\notag
\end{align}
Each successive quotient $$Q_i=(J_{g,*}(i)\cap J^\cfg_{g,*}(n,k))/(J_{g,*}(i+1)\cap J^\cfg_{g,*}(n,k)),$$ for $i=1,\ldots, n$ of these inclusions is naturally a subgroup of the quotient $$R_i=(J_{g,*}(i)\cap J^\cfg_{g,*}(n,i+1))/J_{g,*}(i+1);$$
this is because the map $Q_i\to R_i$ induced by subgroup inclusions has kernel 
$$((J_{g,i}(i)\cap J^\cfg_{g,*}(n,k))\cap J_{g,*}(i+1))/(J_{g,*}(i+1)\cap J^\cfg_{g,*}(n,k)),$$
which is the trivial group. On the other hand, the quotient $R_i$ is precisely the kernel of $$\comp_i:J_{g,*}(i)/J_{g,*}(i+1)\to J^\cfg_{g,*}(n,i)/J^\cfg_{g,*}(n,i+1),$$ which, in Theorem \ref{thm:differenceofJohnsonimages}, we proved is a finitely generated abelian group with an $\Sp_{2g}(\ZZ)$ action, that becomes algebraic of weight $\le 3(i+1)-2(n+1)-1$ after tensoring with $\QQ$. Its subgroup $Q_i$ must, then, also have these properties.

Then for each $1\le i\le k-1$, we have the exact sequence of groups  \begin{small}
    $$1\to (J_{g,*}(i+1)\cap J^\cfg_{g,*}(n,k))/J_{g,*}(k+1)\to (J_{g,*}(i)\cap J^\cfg_{g,*}(n,k))/J_{g,*}(k+1)\to Q_i\to 0,$$
\end{small}
on whose terms the group $\Gamma_{g,*}$ acts by conjugation and the maps are equivariant with respect to it.  The middle term of the sequence has a $\Gamma_{g,*}$-equivariant  to $U:=J^\cfg_{g,*}(n,k)/J^\cfg_{g,*}(n,k+1)$, a finitely generated free abelian group, which after tensoring with $\QQ$, is algebraic of weight $> 3(k+1)-2(n+1)-1$; then the only $\Sp_{2g}(\ZZ)$-equivariant map $Q_i\otimes \QQ\to U\otimes \QQ$ is the trivial map. From Lemma \ref{lem:extensionbyirrelevantrep}, the image of $J_{g,*}(i+1)\cap J^\cfg_{g,*}(n,k)/J_{g,*}(k+1)$ in $U$ is a finite index subgroup of the image of  $(J_{g,*}(i)\cap J^\cfg_{g,*}(n,k))/J_{g,*}(k+1)$. Iterating for $i=1,\ldots, k-1$, the image of $J_{g,*}(k)/J_{g,*}(k+1)$ in $J^\cfg_{g,*}(n,k)/J^\cfg_{g,*}(n,k+1)$ is a finite index subgroup of the image of $J^{\cfg}_{g,*}(n,k)/J_{g,*}(k+1)$. From equation \ref{eq:comp_kfactoring}, the latter surjects $J^\cfg_{g,*}(n,k)/J^\cfg_{g,*}(n,k+1)$.
\end{proof}

Recalling that $J^\cfg_{g,*}(n,n)$ coincides with $J^\cfg_{g,*}(n)$, we can finally describe the quotient $J^\cfg_{g,*}(n)/J_{g,*}(n)$.

\begin{theorem}\label{thm:virtualcofiltration}
    There is a sequence 
    \begin{align*}
            J^\cfg_{g,*}(n)/J_{g,*}(n)=J^\cfg_{g,*}(n,n)/J_{g,*}(n)&\xrightarrow{f_n} J^\cfg_{g,*}(n,n-1)/J_{g,*}(n-1) \\
            &\xrightarrow{f_{n-1}} J^\cfg_{g,*}(n,n-2)/J_{g,*}(n-2)\\
            &\xrightarrow{f_{n-2}} \cdots,
    \end{align*}
    of maps with finite cokernel, so that, for each $k$, the kernel of $f_k$ is a finitely generated free abelian group containing $(J_{g,*}(k-1)/J_{g,*}(k))^{\le 3k-2(n+1)-1}$. 
    
    If $g\ge \max(n,k+1)$, then this containment is of finite index.
\end{theorem}
\begin{proof}
    From a general fact from group theory (see Lemma \ref{lem:grouptheoryexchange} below), the map $f_k$ has the same kernel and cokernel as the map $$J_{g,*}(k-1)/J_{g,*}(k)\to J^\cfg_{g,*}(n,k-1)/J^\cfg_{g,*}(n,k)$$
    where we have exchange the $J_{g,*}(k-1)$ and $J^\cfg_{g,*}(n,k)$. The latter map finite cokernel from Theorem \ref{thm:finitecokernel}. Its kernel contains $(J_{g,*}(k-1)/J_{g,*}(k))^{\le 3k-2(n+1)-1}$ and, if furthermore $g\ge \max(n,k+1)$, this containment is of finite index from Theorem \ref{thm:differenceofJohnsonimages}.
\end{proof}

\begin{lemma}\label{lem:grouptheoryexchange}
    Suppose we have the inclusions $C\subset B_i\subset A$, for $i=1,2$, of normal subgroups. Then the natural maps $B_1/C\to A/B_2$ and $B_2/C\to A/B_1$ have the same kernel and cokernel.
\end{lemma}
\begin{proof}
    Both kernels are $(B_1\cap B_2)/C$ and both cokernels $A/(B_1B_2)$.
\end{proof}

\bibliographystyle{amsalpha}

\bibliography{biblio.bib}

\end{document}